\documentclass{amsart}
\usepackage{amssymb}
\usepackage{color}
\usepackage{float}
\usepackage[all,cmtip]{xy}
\usepackage{graphicx}
\usepackage{hyperref}
\usepackage{cancel}
\usepackage{pstricks}
\usepackage{mathabx}

\makeatletter
\newsavebox{\@brx}
\newcommand{\llangle}[1][]{\savebox{\@brx}{\(\m@th{#1\langle}\)}%
  \mathopen{\copy\@brx\kern-0.5\wd\@brx\usebox{\@brx}}}
\newcommand{\rrangle}[1][]{\savebox{\@brx}{\(\m@th{#1\rangle}\)}%
  \mathclose{\copy\@brx\kern-0.5\wd\@brx\usebox{\@brx}}}
\makeatother

\newcommand{\Ccal}{{\mathcal{C}}}

\newcommand{\RR}{\mathbb{R}}

\newcommand{\rot}{\operatorname{Rot}}

\newtheorem{theorem}{Theorem}
\newtheorem*{theorem*}{Theorem}
\newtheorem{proposition}{Proposition}[section]
\newtheorem{lemma}[proposition]{Lemma}
\newtheorem{conjecture}[proposition]{Conjecture}

\newtheorem{corollary}[theorem]{Corollary}

\theoremstyle{definition}
\newtheorem{definition}[proposition]{Definition}

\theoremstyle{remark}
\newtheorem{remark}[proposition]{Remark}

\numberwithin{equation}{section}

\newcommand{\bparagraph}[1]{\vspace{1em}\noindent\textbf{#1}.\hspace{0.2cm}}

\begin{document}

\title[]{On the causal discontinuity of Morse spacetimes}

\author{Lucas Dahinden, Liang Jin}

%    General info

\date{\today}

\begin{abstract}
Morse spacetime is a model of singular Lorentzian manifold, built upon a Morse function which serves as a global time function outside its critical points. The Borde--Sorkin conjecture states that a Morse spacetime is causally continuous if and only if the index and coindex of critical points of the corresponding Morse function are both different from 1. The conjecture has recently been confirmed by Garcia Heveling for the case of small anisotropy and Euclidean background metric. Here, we provide a complementary counterexample: a four dimensional Morse spacetime whose critical point has index 2 and large enough anisotropy is causally discontinuous and thus the Borde--Sorkin conjecture does not hold. The proof features a low regularity causal structure and causal bubbling.
\end{abstract}

\maketitle

\tableofcontents

% ----------------------------
\section{Introduction}\label{sec:intro}
% ----------------------------

\bparagraph{Historical Background}
In general relativity, spacetime is called globally hyperbolic if it admits a Cauchy surface, i.e., a spatial slice of the universe through which every inextendable causal curve passes exactly once. Global hyperbolicity has turned out to be a very useful property as it allows to understand Einstein's equations as a Cauchy problem. By a result of R.\ Geroch \cite{Ge2}, every globally hyperbolic spacetime is topologically the product of a Cauchy surface $\Sigma$ (space) and the real numbers $\RR$ (time). Later, it was proved that the topological splitting is smooth and the metric has an orthogonal splitting $g(x,t)=h(x,t) - f(x,t)dt^2$, where $f$ is a positive function and for each $t\in\RR, h(t,\cdot)$ is a Riemannian metric on~$\Sigma$, see~\cite{BS}. In particular, any two Cauchy surfaces of the spacetime are diffeomorphic to each other. This fact suggests that spatial topology change is incompatible with the classical globally hyperbolic model of the universe. 

\vspace{1em}
However, in the modern theory of Quantum Gravity spatial topology change may be rather natural. J.Wheeler \cite{Wheeler} argued in the 1950s that quantum fluctuations of the spacetime should produce topological modification. Because of the result of Geroch, this requires models of spacetime with non-standard features. One of the first approaches that was extensively studied was by A.Anderson and B.DeWitt~\cite{AdW} who investigated a quantum field theory for a scalar mass field on the trousers spacetime. Their non-standard approach was to allow the metric to be hermitian in a neighborhood of the crotch in order to resolve the topology change for transitions between the waist band $S^1$ and the foot holes $S^1\cup S^1$. However, they found that this approach leads to infinite energy spikes and infinite particle generation. 

\vspace{1em}
A different non-standard approach, introduced by P.Yodzis~\cite{Yodzis}, were Lorentz cobordisms. Later, they were studied under the name \textit{Morse spacetime} by R.D.Sorkin and J.Louko in \cite{Sor1} and \cite{LS}, and this is the name we adapt here, see Section~\ref{subsec:MorseSpacetimes} for more details. For now, let us note that, like Morse theory, Morse spacetimes feature critical points. At these points the Lorentz metric degenerates and the topology of space changes.

\vspace{1em}
Some Morse spacetimes lack causal continuity (a notion introduced by S.Hawking and R.Sachs~\cite{HS}): The simplest Morse function on the trousers spacetime has one unique critical point in the crotch. If any relativistic PDE problem is set up as a Cauchy problem with initial conditions in the legs, the solution will propagate in a standard fashion up the legs. But there is no a priori reason the solutions will extend continuously over the crotch because below the crotch the two legs cannot ``see each other''. This loss of causal continuity happens always if the stable or unstable manifold of the critical point is disconnected, i.e., at index or coindex 1 critical points~\cite{DGS}. This is natural from the Morse point of view, but it may be surprising from the Lorentz point of view since in classical Lorentz geometry global hyperbolicity implies causal continuity, so the introduction of Morse singularities perturbs the classical causal hierarchy described in~\cite{MS}. 

\vspace{1em}
Sorkin conjectured that causally continuous Morse spacetimes present physically meaningful models for topology change~\cite{Sor1}. From a Morse theory point of view it is clear that if the index and coindex of the critical point is different from 1, then the stable and unstable manifolds are connected, so in this case the topological reason for the causal discontinuity disappears. Borde and Sorkin conjectured that for Morse spacetimes there is no reason for causal discontinuity beyond the disconnectivity of stable and unstable manifolds:
\begin{conjecture}[Borde--Sorkin]
    A Morse spacetime is causally continuous if and only if all critical points of the Morse function have index and coindex different from 1.
\end{conjecture}
A first step towards the proof of this conjecture was made in~\cite{BDGSS}, where it was shown that the specific Morse function on $\RR^n\oplus\RR^m$ with $n,m\geq 2$ given by $f(\mathbf x,\mathbf y)=-\|\mathbf x\|^2+\|\mathbf y\|^2$ (and Euclidean background metric) results in a causally continuous spacetime. Recently, Garcia Heveling proved that the Borde--Sorkin conjecture is true for a much larger family of Morse spacetimes~\cite{GH}: the Borde--Sorkin conjecture is true under the assumption that the critical points are not too anisotropic, i.e., that coefficients of the Morse function have ratio close to $\pm1$ (and that the background metric is Euclidean). 

On the other hand, Garcia Heveling conjectured that in case of too much anisotropy the Borde--Sorkin conjecture is wrong: he thought that one could find a dynamical separation that - similar to the topological disconnectivity of the unstable manifold in the trousers spacetime - creates causal discontinuity, see~\cite[Section 3.2]{GH} or Subsection~\ref{subsec:TwoDim} below. We indeed produce here an (strongly anisotropic) example of four dimensional Morse spacetime whose critical point has index and coindex 2, but is causally discontinuous, thus contradicting the Borde--Sorkin conjecture.

\begin{theorem}\label{thm:main}
Consider a Morse spacetime that admits a local chart $(\RR^4,h,f,\zeta=2)$ where the metric is Euclidean $h=\sum^{2}_{i=1}[(dx_i)^2+(dy_i)^2]$ and the Morse function is
\begin{equation}\label{ex-mf}
f(x)=\frac{1}{2}[-(x_1)^2-b(x_2)^2+(y_1)^2+(y_2)^2].
\end{equation}
There is a constant $b_0>0$ such that this spacetime is causally discontinuous for all $b>b_0$. More precisely, the points $$q=(\sqrt2+1 ,0,1,0),\quad p=(-(\sqrt 2-1),0,1,0)$$ are points of causal discontinuity. 
\end{theorem}
\begin{proof}[Proof of Theorem~\ref{thm:main}]
    By Lemma~\ref{lem:IminusExclusion} $I^-(q)\nsubseteq I^-(p)$ and by Lemma~\ref{lem:IplusInclusion} $I^+(p)\subseteq I^+(q)$, so $p$ and $q$ violate reflectivity and are thus points of causal discontinuity. 
\end{proof}
\begin{corollary}
    The Borde--Sorkin conjecture is wrong.
\end{corollary}
\vspace{0.5em}
\begin{theorem}~\label{thm:B8}
    For $b=8$, the assertion of Theorem~\ref{thm:main} is also true\footnote{We do not prove that $8$ is `large enough' in the sense of Theorem~\ref{thm:main}. Nevertheless, we deem it unlikely that the optimal constant $b_0$ is larger than $8$.}.
\end{theorem}

Lemma~\ref{lem:IplusInclusion} is true for all $b$. Lemma~\ref{lem:IminusExclusion} is true for $b$ large enough (thus the constant $b_0$). The proof of Lemma~\ref{lem:IminusExclusion} involves a projection to a two dimensional subspace, where the projected causal structure has low regularity (1/2-Hölder) along the boundary of the future of the critical point. This generates causal bubbling, similar to Example 1.11 in~\cite{CG}. For $b$ small enough, this causal bubble does not disrupt causal continuity, but for $b$ large enough it manages to break causal continuity. 

\bparagraph{Possible generalizations} In Theorem~\ref{thm:main} we make a few assumptions: We assume $\zeta=2$, we assume specific values for the Morse function, and we assume that $h$ is the Euclidean metric. We only aim to construct a counter example, thus we do not pursue maximal generality. We expect that the assumptions could be relaxed, though the proof would become more difficult:

The assumption on $\zeta$ is taken for computational ease. In Lemma~\ref{lem:zetatheta} below, we show that this assumption amounts to require that the causal cones having an opening angle of $\frac \pi4$ with resepct to the background Euclidean metric. 

The assumption on the coefficients of the Morse function being $-1,-b,1,1$ are taken for geometric clarity. In the first two coefficients we mainly care about large anisotropy: $\frac {-b}{-1}$ should be large, as Garcia Heveling proved that for small anisotropy the Morse spacetime is causally continuous. The second two coefficients are chosen to be equal so that we can take a quotient by rotational symmetry. 

The assumption on $h$ can be understood as an approximation: the Morse lemma allows us to take coordinates such that $h(0)$ is Euclidean at the origin, but $h$ may have curvature. The causal structure is defined by the angle to the gradient $\nabla f$, so functions with proportional gradient fields produce the same causal structure. In particular, $f$ generates the same cone structure as $c^2 f$ for any nonzero constant $c$. We use this effect and `zoom in' by dilating the coordinates by a factor $1/c$ and by rescaling $f$ by $c^2$. Since $f$ in Morse normal form is 2-homogeneous, this zooming in leaves $f$ invariant and only changes $h$. For $c\to\infty$, the rescaled $h$ approximates its linearization, i.e., the Euclidean metric. The Morse chart in Theorem~\ref{thm:main} describes thus the linearization of more general situations. 

We expect that the first two assumptions can be omitted, though our geometric approach will require modification and computations will become more complicated. 
To us the third assumption is the most interesting one: It is not obvious to us that a causally discontinuous linearization proves causal discontinuity. Our algebraic constructive method fails here and more general methods are required. Thus, we would be very interested in a proof of causal discontinuity for non-flat $h$.

\subsection{Outline of the paper} In Section~\ref{sec:definitions} we introduce the neccessary notions: In Subsection~\ref{subsec:MorseSpacetimes} Morse spacetimes and in Subsection~\ref{subsec:CausalContinuity} causal continuity. In Subsection~\ref{subsec:TwoDim} we also discuss the two dimensional situation that led Garcia Heveling to the conjecture that the Borde-Sorkin conjecture may be wrong.

In Section~\ref{sec:proof} we begin by outlining the geometric setup. In Subsection~\ref{subsec:PoinsOfDiscontinuity} we state the two Lemmas~\ref{lem:IplusInclusion} and~\ref{lem:IminusExclusion} that together prove Theorem~\ref{thm:main}. We directly prove Lemma~\ref{lem:IplusInclusion} and we outline the proof of Lemma~\ref{lem:IminusExclusion}. The rest of the section is then concerned with the proof of this second lemma. 

In Section~\ref{sec:B8} we prove Theorem~\ref{thm:B8} by showing that the construction of the previous section also works for $b=8$.

The Appendix~\ref{sec:appendix} contains a proof that gradient flow lines are length maximising geodesics. This fact is not used in the other results, but we have included the proof since we have not found it elsewhere and since it may be of independent interest.

\subsection*{Acknowledgements} The first author was funded by the Deutsche Forschungsgemeinschaft (DFG, German Research Foundation) – Project-ID 281071066 – TRR 191. The second author is supported in part by the National Natural Science Foundation of China (Grant No. 12171096, 12371186). We want to thank Stefan Suhr for making us aware of the project and  Leonardo Garcia Heveling for the discussion of the problem.

% ---------------------------- ----------------------------
% ---------------------------- ----------------------------
% ---------------------------- ----------------------------
% ---------------------------- ----------------------------
% ---------------------------- ----------------------------
% ---------------------------- ----------------------------
\section{Preliminaries}\label{sec:definitions}

\subsection{Morse Spacetimes}\label{subsec:MorseSpacetimes} A Morse spacetime $(M^{n+m},h,f,\zeta)$ is constructed from the following data: a compact, connected $(n+m)$-dimensional smooth Riemannian manifold $(M^{n+m},h)$, possibly with boundary, a Morse function $f$ and a real number $\zeta>1$. Recall that a smooth function $f:M\rightarrow\RR$ is said to be Morse if the zeros of the section $df$ of $T^*M$ (i.e., the critical points) are transversely cut out. This is equivalent to saying that the Hessian of $f$ is non-degenerate at critical points. It follows from the Morse lemma \cite[Theorem 1.3.1]{AD} that near each critical point $p$, one can find a coordinate system $\{z_i\}_{1\leq i\leq n+m}$ such that
\begin{equation}
f(\mathbf z)=f(p)+\frac{1}{2}\sum^{n}_{i=0}a_i z_i^2,
\end{equation}
where $\mathbf z=(z_1,...,z_{n+m})$ and $a_i$ are nonzero constants. In particular, Morse functions only have isolated critical points. We assume that there are no critical points on the boundary $\partial M$. Since all our considerations are local around critical points we will allow $M=\RR^{n+m}$ with the assumption that all data are pullbacks along a smooth embedding into a compact manifold.  

\begin{definition}\label{def:g}
    We define the \emph{Morse--Lorentzian metric} on $M$
$$g(\mathbf z)= \|df(\mathbf z)\|^2 h(\mathbf z) - \zeta df(\mathbf z)^{\otimes 2}.$$
\end{definition}

This is clearly a symmetric bilinear form. If $df(\mathbf z)=0$, then $g(\mathbf z)=0$. For any other point $\mathbf z$ the tangent space splits into gradient and tangent space of the level set $T_{\mathbf z}M=\nabla f(\mathbf z)\RR \times\ker df(\mathbf z)$. The requirement $\zeta>1$ ensures that $g(\mathbf z)(\nabla f(\mathbf z),\nabla f(\mathbf z))<0$, and on the level sets $g(\mathbf z)|_{\{f\equiv c\}} = \|df(\mathbf z)\|^2 h(\mathbf z)$ is clearly positive definite, so $g$ is indeed a Lorentz metric away from the critical points of $f$. 

\begin{remark}
    This construction is more general than it may seem on first inspection: One may also start with $g$ and a function $f$ such that
    \begin{itemize}
    \item $g(\mathbf z)=0$ iff $df(\mathbf z)=0$, the zeroes of $df$ are transverse and $\|g\|$ is quadratic around the zeroes,
    \item $g$ is Lorentzian away from the critical points of $f$,
    \item $\ker df$ is spacelike.
    \end{itemize}
    Then since $M$ is compact, there is a number $\zeta$ large enough such that $h=\frac 1{\|df\|^2}(g+\zeta df^{\otimes 2})$ is a Riemannian metric that continuously extends over the critical points and $(M,h,f,\zeta)$ is Morse--Lorentz with Morse--Lorentz metric $g$.
\end{remark}

We will use the level sets of $f$ as spatial slices. They change topology at critical points according to surgery theory: passing over a critical point of index $n$ amounts to attaching a $(n,m)$-handle to the slice.

We specialize to the case of a Morse chart $(\mathbb{R}^{n+m},h)$ where $h$ is Euclidean
\[
h=\sum^{n+m}_{i=1}dz_i^2
\]
and the Morse function $ f:\mathbb{R}^{n+m}\rightarrow\mathbb{R}$ is normalized to have critical value 0 and has the standard form
\begin{equation}\label{def:f}
f(\mathbf z)=\frac{1}{2}\sum_{i=1}^{n+m}a_i z_i^2.
\end{equation}
W.l.o.g, we assume $n_s=n$ of the $a_i$ are negative and $n_u=m$ of the $a_i$ are positive (and none are 0 because of non-degeneracy). By relabeling we achieve
\[
a_1\leq \ldots\leq a_{n}<0<a_{n+1}\leq\ldots \leq a_{n+m}.
\]
By denoting $A=\operatorname{diag}(a_1,\ldots,a_{n+m})$ we can also write 
\[
f(\mathbf z)=\frac12 \mathbf z^\intercal A\mathbf z, \quad \nabla f(\mathbf z)= A\mathbf z.
\]
For the positive gradient flow of $f$: $0$ is the only fixed point, the subspace $\RR^{n}\times\{0\}$ is the stable manifold, where we denote its coordinates by $\mathbf x=(x_1,x_2,...,x_n)$ and its orthogonal complement $\{0\}\times\RR^{m}$ is the unstable manifold, where the coordinate is denote by $\mathbf y=(y_1,...,y_m)$, thus $\mathbf z=(\mathbf x,\mathbf y)$ and every line not converging to the origin for positive or negative time is a hyperbola.
\begin{remark}\label{rem:switchConvention}
There is an unfortunate conflict of conventions: In Morse theory usually one flows down and thus perceives the negative gradient flow of $f$ as the natural direction. In Lorentzian geometry, $f$ is time, which naturally increases, so the positive gradient flow is the natural direction. Here, we follow the Lorentz convention: the indices $u$ and $s$ stand for unstable and stable with the the positive gradient flow in mind. 
\end{remark}

The lightlike vectors based at any point form a cone with axis $\nabla f$ and constant angle (measured by $h$): 
\begin{lemma}\label{lem:zetatheta}
If $g(v,v)=0$, then $\cos(\angle(v,\nabla f))=\pm\zeta^{-\frac12}$. We denote the two possibilities by $\angle(v,\nabla f)=\theta$ and $\angle(v,\nabla f)=\pi-\theta$, where $\theta\in(0,\frac \pi 2)$. If $\zeta = 2,$ then $\theta=\frac\pi4$.
\end{lemma}
\begin{proof}
    Notice that 
    \begin{align*}
        0=g(v,v)=\langle \nabla f,\nabla f\rangle\langle v,v\rangle-\zeta \langle \nabla f,v\rangle^2,
    \end{align*}
which implies 
\[
\frac{\langle v, \nabla f\rangle^2}{|v|^2|\nabla f|^2} = \zeta^{-1}.
\] 
Taking the square root then results in 
\[
\cos(\angle(v,\nabla f))=\frac{\langle v, \nabla f\rangle}{|v||\nabla f|} = \pm\zeta^{-\frac12}.
\] 
\end{proof}

The spacelike vectors are thus the vectors $v$ with $\angle(\nabla f,v)\in(\theta, \pi-\theta)$. Such a Morse--Lorentz spacetime is naturally time oriented: The timelike cone splits into the two components of vectors $v$ with $\angle(\nabla f, v)<\theta$, which we choose for our positive orientation and $\angle(\nabla f, v)>\pi-\theta$, which then is the negative orientation. We denote these cones by
$$C^+(\mathbf z)=\{v\in T_\mathbf z\RR^n\mid \angle(v,\nabla f) <\theta\},\quad C^-(\mathbf z)=\{v\in T_\mathbf z\RR^n\mid \angle(v,\nabla f)>\pi-\theta\}.$$ 
Note that at the critical point $g_0=0$, thus every vector is lightlike. In particular $0$ admits no time-like vector. 

A positive (negative) time-like curve $\gamma$ is a regular curve not crossing zero such that $\angle(\nabla f (\gamma(t)), \dot\gamma(t))< \theta$ (resp.~$>\pi-\theta$) for all $t$. If there is a positive (negative) time-like curve from $q$ to $p$, then we say that $p$ lies in the future (past) of $q$.
\begin{definition}\label{def:FuturePast}
    The \emph{future of $q$} $I^+(q)$ (and the past $I^-(q)$) is the set of points that lie in the future (past) of $q$. 
\end{definition}
\begin{remark}
    There is also the notion of causal future and causal past, usually denoted by $J^\pm$, based on the existence of non-spacelike curves connecting the points. It shall be noted - though the set of non-spacelike vectors is the closure of the set of timelike vectors - that $J^\pm$ is not generally closed, and that even if it is closed the closure of $I^\pm$ is not generally $J^\pm$. We will focus on $I^\pm$.
\end{remark}

% ---------------------------- ----------------------------
% ---------------------------- ----------------------------
% ---------------------------- ----------------------------
% ---------------------------- ----------------------------
% ---------------------------- ----------------------------
% ---------------------------- ----------------------------
\subsection{Causal continuity}\label{subsec:CausalContinuity}

We begin by notice that the definition of $I^\pm$ and $J^\pm$ of a point gives rise to the chronological and the causal relation. Many `good' properties that a Lorentz manifold can have can be reformulated as causality properties, i.e., properties of these relations. In classical Lorentz geometry there is a hierarchy, see~\cite{MS} of causality properties, where the higher ones imply the lower ones.
\begin{itemize}
    \item Globally hyperbolic
    \item Causally simple
    \item Causally continuous
    \item Stably causal
    \item Strongly causal
    \item Distinguishing
    \item Causal
    \item Chronological
    \item Non-totally vicious
\end{itemize}
For a comprehensive introduction to such a causal ladder in  Lorentzian geometry, we refer the readers to the stanard textbook~\cite{BEE} or the recent survey paper~\cite{M}. Recently this hierarchy has also been established for Lorentz length spaces (see~\cite{APS}), which is a modern setup that has proved to be effective at dealing with low regularity Lorentz manifolds. Since Morse--Lorentz manifolds are not Lorentz, the causality properties must be generalized (see \cite{BDGSS}). The generalized causality properties then do not necessarily satisfy the hierarchy, and in fact Theorem~\ref{thm:main} gives an example when it is not the case: it gives an example that is (generalized) globally hyperbolic but not (generalized) causally continuous.

For Morse--Lorentz manifolds only the top of the hierarchy is interesting, especially global hyperbolicity, causal continuity and stable causality. We will use the following classical definitions, as they admit natural generalizations described below.
\begin{itemize}
    \item \emph{Classical stable causality} is equivalent to the existence of a time function~$f$. 

    Note that this means that the manifold is foliated by spacelike hypersurfaces (the level sets of the time function) through which timelike curves can only pass transversely in one direction.
    \item \emph{Classical causal continuity} means that the sets $I^\pm(q)$ depend outer continuously on the point $q$, see Definition~\ref{def:CausalContinuity}. 
    \item \emph{Classical global hyperbolicity} is equivalent to saying that there is a time function $f$ such that every inextendible causal curve passes through all times, and $\overline I^\pm(q)\cap f^{-1}(t)$ is compact for all $q,t$. 
\end{itemize}

These properties are generalized as follows.

\vspace{1em}
\emph{Morse--Lorentz stable causality:} A Morse--Lorentz spacetime has a time function $f$. Thus, Morse--Lorentz manifolds satisfy a reasonable generalization of stable causality by default. 

\vspace{1em}
\emph{Morse--Lorentz causal continuity:} A timelike curve that converges to a critical point can not be timelikely extended (since the critical point does not admit a timelike vector), which means that $I^\pm(p)=\emptyset$ if $p$ is a critical point. To circumvent the problems arising from this, one imitates the machinery of broken gradient-like flow lines from Morse theory: we allow broken curves, which are pieces of curves where each piece that ends at a critical point is followed by a next piece that starts at this critical point. From now on we denote by $I^\pm(p)$ the broken future/past of $p$. Note that this way limits of timelike curves are broken non-spacelike curves, and broken timelike curves are always limits of timelike curves. Our interest is whether this correspondence can be completed: Our main Lemmas~\ref{lem:IminusExclusion} and~\ref{lem:IplusInclusion} basically show that a broken non-spacelike curve is a limit of timelike curves from one side, but not from the other. 

The following definitions (which are all also equivalent definitions of causal continuity in the case of Lorentz manifolds) are equivalent, see \cite{BDGSS}.

\begin{definition}\label{def:CausalContinuity}
    A Morse--Lorentz manifold is causally continuous if it is distinguishing (i.e., if $I^+(p)=I^+(q)$ or $I^-(p)=I^-(q)$ then $p=q$) and if one of the following equivalent conditions holds:
    \begin{itemize}
        \item The timelike pasts and futures are outer continuous: If a compact set $K$ does not intersect $\overline{I^\pm(q)}$, then there is an open neighborhood $U$ of $q$ such that $\forall p\in U: \; \overline{I^\pm (p)}\cap K = \emptyset$.
        \item Approximation: The future (past) of a point is the collective future (past) of its past (future): $I^+(q) = \uparrow I^-(q)$ and same statement for -.
        \item Reflecting: $I^+(p)\subseteq I^+(q)\Leftrightarrow I^-(q)\subseteq I^-(p)$.   
    \end{itemize}    
\end{definition}

We will use the third property to show causal discontinuity: we look for two points $p,q$ such that $I^+(p)\subseteq I^+(q)$ but $I^-(q)\nsubseteq I^-(p)$.

\vspace{1em}
\emph{Morse--Lorentz global hyperbolicity:} For global hyperbolicity we have the same problem of timelike curves ending at critical points, and we also solve it by incorporating piecewise timelike curves. Then, the important remaining property of global hyperbolicity is $\overline I^\pm(q)\cap f^{-1}(t)$ is compact for all $q,t$. This is a very natural assumption in Morse theory, often realized by demanding that the level sets be compact. It is also true for the standard Morse--Lorentz chart that we investigate here. 

%-------------------------------------------------------------------
%-------------------------------------------------------------------
%-------------------------------------------------------------------
%-------------------------------------------------------------------
%-------------------------------------------------------------------
%-------------------------------------------------------------------
%-------------------------------------------------------------------
%-------------------------------------------------------------------
%-------------------------------------------------------------------
%-------------------------------------------------------------------
%-------------------------------------------------------------------

\subsection{Two dimensional situation}\label{subsec:TwoDim}
We quickly reproduce Example 1 from~\cite[Section 3.2]{GH} which led Garcia Heveling to the the conjecture that the Borde--Sorkin conjecture is wrong for large anisotropy.

We consider the local chart $\RR^2$ equipped with the Euclidean metric, we denote $\mathbf x =(x_1,x_2)$, and as Morse--Lorentz function we set  $$f(\mathbf x) = -\frac12(x_1^2+bx_2^2).$$ 
This is a good description for all standard Morse--Lorentz functions since $f$  induces the same cone structure as $\lambda f$ for $\lambda>0$ (and it flips the time orientation for $\lambda<0$).

If $b<0$, then the index (and coindex) is 1. This is the local model around the critical point of the ``trousers spacetime'' that was studied already in\cite{AdW} and causal discontinuity essentially stems from the fact that the sublevel set $\{f<0\}$ is disconnected.

For us, the interesting case is $b>0$. By swapping coordinates and renormalizing we assume $b\geq 1$. In this example, time has a global maximum of 0 and the spacetime can be seen as a parabolic hat\footnote{In some sources, this is referred to as the yarmulke spacetime.}. Then, $\nabla f(\mathbf x)=-\begin{pmatrix}1&0\\0&b\end{pmatrix}\mathbf x$. The reason that the two dimensional situation is so easy to understand is that the set of timelike positively oriented vectors lie in a cone whose boundary consists of the two rays $\RR_{\geq0}A_\pm\mathbf x$ spanned by the gradient vector tilted by the opening angle to the left or right $A_\pm\mathbf x$, where $$A_\pm=-\begin{pmatrix}\cos\pm\theta&-\sin\pm\theta\\\sin\pm\theta&\cos\pm\theta\end{pmatrix}\begin{pmatrix}1&0\\0&b\end{pmatrix}.$$ 
Thus, to understand timelike curves it suffices to understand two matrices and `anything in between' is timelike.
The characteristic polynomial of the matrix $A_\pm$ is $b-\lambda\cos\theta (b+1)+\lambda^2$ with discriminant $\cos^2\theta (b+1)-4b$. The discriminant is non-negative if
$$\cos\theta\geq \frac{GM(1,b)}{AM(1,b)}.$$ 
The right hand side is always $\leq 1$ because of the $AM-GM$ inequality, and equality only holds for $b=1$. In that equality case, the vector fields $A_\pm\mathbf x$ integrate to logarithmic spirals. For every angle $\theta$ there is $b$ large enough such that the inequality is strictly satisfied and thus there are (two different) real eigenvalues and consequently there are two different real eigenspaces for each of the two Matrices $A_\pm$. These are breaking points: on one side of the Eigenspaces positive timelike vectors can cross radial lines, on the other side they can not. Thus, the two pairs of Eigenspaces cut the plane in regions as depicted in Figure~\ref{fig:2DCones}, where positive timelike curves can only exit the light blue regions and only enter the red regions. In particular, it is not possible to travel in a lightlike manner from one red region to the other. This means that the past of the critical point splits into different regions with restricted manouvrability like in the trousers spacetime, though the separation is much less clear (here, the separation is dynamical and not topological).

\begin{figure}
    \centering
    \includegraphics[width=0.5\linewidth]{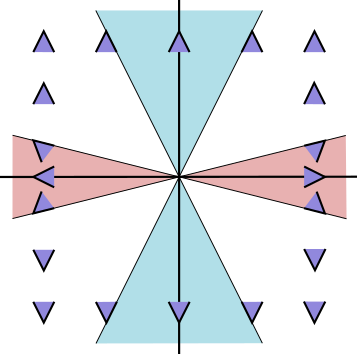}
    \caption{The dark blue triangles indicate the positive timelike cones at selected points. For large $b$, the gradient vector field is almost vertical away from a small neighborhood of the $x_1$-axis. Positive lightlike curves can only exit the light blue region and only enter the red region.}
    \label{fig:2DCones}
\end{figure}

Garcia Heveling conjectured for a four dimensional example that admits a two dimensional slice that looks like the example above that this lack of manouvrability in the past of the two dimensional slice can be utilized to prove causal discontinuity in the entire four dimensional spacetime. For more details on his conjecture, see Section 3.2 in~\cite{GH}. The difficulty is to extend the separating eigenspaces to separating hypersurfaces in four dimensional space. This is the main step in the proof of Theorem~\ref{thm:main}.

\section{Proof of Causal discontinuity} \label{sec:proof}

In Subsection~\ref{subsec:PlaneProjection}, we will first understand the four dimensional problem better by projecting to an affine plane. Then we establish some geometric facts of this projection in Subsection~\ref{subsec:DefineOmega}. In Subsection~\ref{subsec:PoinsOfDiscontinuity} we identify two points $q,p$ that violate causal continuity. This is confirmed in two parts: $I^+(p)\subseteq I^+(q)$ (Lemma~\ref{lem:IplusInclusion}) but $I^-(q)\nsubseteq I^-(p)$ (Lemma~\ref{lem:IminusExclusion}). The first lemma is quickly proved by a push-up proof. The second lemma requires more effort: essentially it is done by finding a barrier, which we do the remainder of this section. 

\subsection{Projection to the affine plane}\label{subsec:PlaneProjection} We project the four dimensional space to a plane in two steps: first we divide out the angular coordinate of $\mathbf{y}=(y_1,y_2)$, thus reducing the situation to $\RR^2\times\RR_{>0}$; then we will project radially to the affine plane $y=1$. The causal cones on the projections will be the the projections of the causal cones. In this way, every causal curve in the original cone fields is projected to a causal curve in the projected cone fields.

\bparagraph{Reducing $\mathbf y$ to its radial coordinate} 
We denote the coordinates of the second $\RR^2$-factor as $\mathbf{y}=(y_1,y_2)$ and its polar coordinates as $(y= |\mathbf{y}|, \varphi_y)$.
We denote the rotation in the $\mathbf{y}$-coordinate by $\rot_\varphi:\RR^2\to\RR^2;\,\,\mathbf{y}=(y_1, y_2)\mapsto\rot_\varphi(\mathbf{y})=(y_1\cos\varphi-y_2\sin\varphi, y_1\sin\varphi+y_2\cos\varphi)$. It is easily seen from the formula \eqref{ex-mf} that $f$ is independent of $\varphi_y$, thus the cone fields $C^\pm$ admit rotational symmetry: a vector $(\mathbf v,\mathbf w)$ at $(\mathbf x,\mathbf y)$ is timelike if and only if the rotated vector $(\mathbf v,D\rot_\varphi \mathbf w)$ at the rotated point $(\mathbf x,\rot_\varphi \mathbf y)$ is timelike. 

\vspace{1em}
We omit points with $\mathbf y=0$ and consider the projection onto the radial coordinate
$$\varrho\colon \RR^4\setminus(\RR^2\times\{0\})\to \RR^2\times \RR_{>0};\quad (\mathbf x,\mathbf y)\mapsto (\mathbf x,y= |\mathbf{y}|)$$
which has an (Euclidean) isometric right inverse $$\iota\colon\RR^2\times \RR_{>0}\to\RR^4\setminus(\RR^2\times\{0\});\quad (\mathbf x,y)\mapsto (\mathbf x,y,0).$$
Because of the rotational symmetry of $f$, it descends to the quotient $\RR^2\times \RR_{>0}$ to the restricted function $\iota^{\ast}f$ and we define the cone fields of $\pm$-oriented timelike vectors as
\begin{align*}
C^+_\varrho(\mathbf x,y)&:=\{v\mid \angle(v,\nabla \iota^{\ast}f) <\theta\},\\
C^-_\varrho(\mathbf x,y)&:=\{v\mid \angle(v,\nabla \iota^*f)>\pi-\theta\}.
\end{align*}
Note that the gradient of $f$ at $\iota(\mathbf x,y)$ is tangent to $\iota(\RR^2\times \RR_{>0})$ and coincides with $D\iota(\mathbf x,y)[ \nabla \iota^*f]$. Since $\iota$ is an isometric inclusion, it is in particular angle preserving. Thus the cones in the projection coincide with the restricted cones:
$$C^\pm (\iota(\mathbf x,y))\cap T_{\iota(\mathbf x,y)}\iota(\RR^2\times\RR_{>0})= D\iota(\mathbf x,y)[C^\pm_\varrho(\mathbf x,y)]. $$
Since $\iota$ is a right inverse of $\varrho$, the cone fields in the projection are precisely 
$$D\varrho\,\,C^\pm(\mathbf x,\mathbf y) = C^\pm_\varrho(\mathbf x,y).$$

Thus, we collect characterizations of vectors $v\in T_p(\RR^2\times\RR_{>0})$. Because of the rotational symmetry of $C^\pm$ it does not matter which preimage point of $p$ we choose.
\begin{itemize}
    \item $v$ is $C_\varrho^\pm$-timelike if there exists a $C^\pm$-timelike preimage vector $D\varrho^{-1}(p)[v]$. Note that not every preimage vector is $C^\pm$-timelike.
    \item $v$ is $C_\varrho^\pm$-lightlike if there does not exist a $C^\pm$-timelike preimage, but there is a $C^\pm$-lightlike preimage. Note that then this $C^\pm$-lightlike preimage is unique (at every preimage point in $\varrho^{-1}(p)$).
    \item $v$ is otherwise $C_\varrho^\pm$-spacelike. Then, the entire preimage $D\varrho^{-1}(p)[v]$ consists of $C^\pm$-spacelike vectors.
\end{itemize}   
For us it is crucial that the preimage of non-timelike vectors consists entirely of non-timelike vectors. This makes the notion of barrier vector below~\ref{def:barrier} meaningful.

The points in $\RR^4$ with $\mathbf y=0$ can be continuously included in the picture by setting $\varrho(\mathbf x,0,0)=(\mathbf x, 0)$, though the differential $D\varrho$ at these points makes no sense. This does not matter much to us, in the next step we will push these points to the projective boundary.

\begin{remark}
Please note that any Morse spacetime with a critical point in $\RR^3$ with index 1 or 2 is automatically causally discontinuous, see~\cite{DGS}. As outlined in the introduction, the essential reason for this is that the attaching (or belt) 0-sphere of the handle is disconnected. However, this does not apply to our situation since the remaining $y$-factor of the space is $\RR_{\geq 0}$, in which we only have half of a 0-sphere, which is indeed connected. 
\end{remark}

\bparagraph{Forgetting homogeneous dilation}
Since both $f$ and $\iota^*f$ are homogeneous polynomials of degree $2$, the cone fields $C^\pm$ as well as $C^\pm_\varrho$ are invariant under homogeneous dilation. Thus, we project further radially from $\RR^2\times\RR_{>0}$ to the $y=1$-plane $\RR^2\equiv \RR^2\times\{1\}$ as
$$\pi(\mathbf x,y)=\frac{\mathbf x}y.$$

\begin{remark} This projection extends continuously to $\RR^2\times\RR_{\geq 0}\setminus \{0\}$ if we extend the target space by adding the projective circle $S^1_\infty$ at infinity and mapping $(\mathbf x,0)$ to $\frac{ \mathbf x}{|\mathbf x|}\in S^1_\infty$. The circle at infinity represents the positive projectivisation of the two dimensional linear situation described in Subsection~\ref{subsec:TwoDim}. This is good to keep in mind, but in the following we do not include the circle at infinity in the description to alleviate notation and since our construction works away from $y=0$. We also notice that this map does \emph{not} extend continuously to $0$. 
\end{remark}

We also intend to project the cones $C_\varrho^\pm$ to the affine plane 
$$
C_1^\pm:=D\pi C_\varrho^\pm.
$$
Recall that $\nabla f(\mathbf x,y)=(-x_1,-bx_2,y)$ and denote the radial vector field by $\rho(\mathbf x, y) = (\mathbf x, y)$. To describe $C_1^\pm$ we regard the projected gradient, which is 
$$D\pi\nabla f(\mathbf x,y) = \frac 1y (\nabla f - \rho) = \frac1y(-2x_1,-(b+1)x_2)$$
since the differential is 
$$D\pi = \begin{pmatrix}
    \frac 1y & 0 & -\frac{x_1}{y^2}\\
    0&\frac 1y  & -\frac{x_2}{y^2}
\end{pmatrix}.$$
By definition, a vector is $C_1^\pm$-timelike if there exists a positive/negative $C^\pm_\varrho$-timelike preimage. It is $C_1^\pm$-lightlike if there does not exist a $C^\pm_\varrho$-timelike preimage, but there is a $C^\pm_\varrho$-lightlike preimage. Otherwise, the vector is $C_1^\pm$-spacelike and all its preimages are $C^\pm_\varrho$-spacelike.

\vspace{1em}
Note that $\ker D\pi(\mathbf x,y) = \RR\cdot \rho$, thus the preimage of a vector $v$ is $D\pi^{-1}v = v+\RR\cdot\rho$. Also, since $C_{\varrho}^\pm$ is a cone, $D\pi^{-1}v$ intersects $C_\varrho^\pm$ (or is tangent to it) iff $D\pi^{-1}\lambda v=\lambda D\pi^{-1} v$ intersects $ C_\varrho^\pm$  (or is tangent to it) for any $\lambda >0$. In other words, the projection of the cone is still a cone and we must investigate whether the plane spanned by $v$ and $\rho$ intersects $ C_\varrho^\pm$ (or is tangent to it). Since $ C_\varrho^\pm$ is determined by the angle to $\nabla f$, we are interested in the angle between the normal $v\times\rho$ of the plane and $\nabla f$: If the angle is closer than $\pi/2-\theta$ to either $\nabla f$ or $-\nabla f$, then vectors from the plane have angle at least angle $\theta$ from both $\nabla f$ and $-\nabla f$, which means that every vector of the plane is $C_\varrho^\pm$-spacelike. Combined with the other inequalities we obtain:

\begin{lemma}\label{lem:CriterionC1}
    Let $0\neq v\in T\RR^2\times \{1\}$. Then, $v$ is, with respect to $C_1^\pm= D\pi C_\varrho^\pm$,
    \begin{itemize}
        \item spacelike if $\angle(v\times\rho,\nabla f) \in[0,\pi/2-\theta)\cup (\pi/2 +\theta,\pi]$,
        \item lightlike if $\angle(v\times\rho,\nabla f) \in\{\pi/2-\theta,\pi/2+\theta\}$,
        \item timelike if $\angle(v\times\rho,\nabla f) \in(\pi/2-\theta,\pi/2+\theta)$.
    \end{itemize}
\end{lemma}

If $\theta =\pi/4$, then the condition to be spacelike  is
$$\angle(v\times\rho,\nabla f) < \pi/4 \mbox{ or } >3\pi/4$$
if we take the cosines, this becomes
$$\frac{\langle v\times\rho, \nabla f \rangle }{|v\times\rho||\nabla f|} > 1/\sqrt 2 \mbox{ or } < -1/\sqrt 2.$$
By taking squares, we consolidate the two cases
$$\left(\frac{ \langle v\times\rho, \nabla f \rangle }{|v\times\rho||\nabla f|}\right)^2>\frac 1{2}.$$
In coordinates, this is the relation 
$$\frac{ (-v_2 x_1 + bv_1x_2 + (v_1x_2 - v_2x_1))^2 }{(v_1^2 +v_2^2 + (v_1x_2-v_2x_1)^2)(x_1^2+b^2x_2^2+1)}>\frac 1{2},$$
i.e., 
$$ 2(-2v_2 x_1 + (b+1)v_1x_2))^2 >{(v_1^2 +v_2^2 + (v_1x_2-v_2x_1)^2)(x_1^2+b^2x_2^2+1)}.$$
Thus (and with anologous considerations for timelike and lightlike) we have proved 
\begin{lemma}\label{lem:CriterionC2}
    Let $\theta=\pi/4$. Let $0\neq v\in T\RR^2\times \{1\}$. Then, $v$ is, with respect to $C_1^\pm= D\pi C_\varrho^\pm$,
    \begin{itemize}
        \item spacelike if \\$2(-2v_2 x_1 + (b+1)v_1x_2))^2 >{(v_1^2 +v_2^2 + (v_1x_2-v_2x_1)^2)(x_1^2+b^2x_2^2+1)}$,
        \item lightlike if \\$2(-2v_2 x_1 + (b+1)v_1x_2))^2 ={(v_1^2 +v_2^2 + (v_1x_2-v_2x_1)^2)(x_1^2+b^2x_2^2+1)}$,
        \item timelike if \\$2(-2v_2 x_1 + (b+1)v_1x_2))^2 <{(v_1^2 +v_2^2 + (v_1x_2-v_2x_1)^2)(x_1^2+b^2x_2^2+1)}$.
    \end{itemize}
\end{lemma}

Note that the non-timelike (i.e., spacelike or lightlike) vectors divide into two separate sets corresponding to the normal vector being close to $\nabla f$ or $-\nabla f$. More precisely, if $v$ is non-timelike and $(v,\rho,\nabla f)$ is a negatively oriented frame, then $\langle v\times\rho,w\rangle<0$ for every positive timelike vector $w$, which means that any $w\in C^+_\varrho$ projects to a vector $D\pi w\in C_1^+$ that form a positively oriented frame $(v,D\pi w)$. This motivates the following definition.

\begin{definition}\label{def:barrier}
    A nonzero vector $v\in T\RR^2\times\{1\}$ is a \emph{barrier vector} if it is non-timelike, i.e., if every timelike vector is transverse to it. Furthermore, it is a \emph{positive (negative) barrier vector} if $\langle v\times\rho,\nabla f\rangle<0$ ($>0$), i.e., if every positively oriented timelike vector $D\pi w$ forms a positive (negative) frame $(v,D\pi w)$ with it. 

    A parametrized curve $\gamma(t)$ with velocity $\dot\gamma(t)$ being a positive (negative) barrier vector for all $t$ is called a barrier.
\end{definition}
\begin{remark} The name barrier is chosen because it controls the flow direction of positive timelike vectors: They can go one way through a barrier, but not the other.

Note that reversing the orientation of a positive barrier makes it a negative barrier. The name is chosen here such that the positively oriented boundary of $\Omega^+$ (defined below) is a positive barrier. 
\end{remark}

We would like to distinguish three different cases:
\begin{itemize}
    \item If $\rho$ is $C^\pm_\varrho$-spacelike, then the cones $C_1^\pm$ are actually two disjoint proper cones.
    \item If $\rho$ is $C^\pm_\varrho$-timelike, then every $v$ admits both positive and negative preimages, i.e., $C_1^+=C_1^-=T\RR^2\times\{1\}$. In this case timelike vectors have complete manouvrability.
    \item If $\rho$ is $C^\pm_\varrho$-lightlike, then both cones $C_1^+$ and $C_1^-$ are half-spaces bounded by the set of lightlike vectors which is a hyperplane $\partial C_1^+=\partial C_1^-$. Thus, the tangent space splits into positive and negative time-vectors and lightlike barrier vectors (and there are no spacelike vectors) $T\RR^2\times\{1\}=C_1^+\cup C_1^-\cup \partial C^\pm_1$.
\end{itemize}

We use these cases to define the following sets of points in $\RR^2\times\{1\}$
\begin{align*}
    \Ccal &= \{\mathbf x\mid \pm\rho(\mathbf x,1)\notin \overline C^\pm_\varrho\},\\
    \Omega^\pm &=\{\mathbf x\mid \rho(\mathbf x,1) \in C^\pm_\varrho\},\\
    \partial\Omega^\pm &= \{\mathbf x\mid \rho(\mathbf x,1) \in \partial C^\pm_\varrho\}.
\end{align*}

Because of continuity of $C^\pm_\varrho$ and $\rho$ we have that $\partial\Omega^\pm$ is indeed the boundary of $\Omega^\pm$, justifying the notation.

The two flavors $\pm$ of $\Omega^\pm$ have the following consequence: Recall that at $\partial\Omega^\pm$ the set of lightlike vectors is a hyperplane.  This hyperplane is the tangent space of $\partial\Omega^\pm$. The signs are chosen such that positively oriented timelike vectors at $\partial\Omega^+$ point inwards $\Omega^+$ and at $\partial\Omega^-$ point outwards $\Omega^-$. 

Recall that a positively oriented  boundary $\gamma$ satisfies that for an outwards normal $n$ the frame $(n,\dot\gamma)$ is positive, or equivalently if $(\dot\gamma, -n)$ is positive. If the positive timelike vector are pointing inwards $\Omega^+$, then this means that the positively oriented boundary $\partial_\Omega^+$ is a positive barrier, making the terminology compatible.

\subsection{Shape of controllable regions}\label{subsec:DefineOmega}

Throughout this subsection, we assume that $\theta=\pi/4$.  

Recall that $\nabla f = (-x_1,-bx_2,1)$ and $\rho=(x_1,x_2,1)$. Then, $\partial\Omega^\pm$ is determined by $\rho$ lying on the boundary of the cone, i.e., by the equation 
\begin{align}
    \cos\angle(\nabla f,\rho) =& \frac{-x_1^2 - bx_2^2+1}{\sqrt{(x_1^2+b^2x_2^2+1)(x_1^2+x_2^2+1)}}\stackrel != \pm\cos \theta = \pm\frac 1{\sqrt2}\nonumber\\
    \Leftrightarrow  2( -x_1^2 - bx_2^2+1)^2 =\label{eq:RhoOnCone}& (x_1^2+b^2x_2^2+1)(x_1^2+x_2^2+1)\nonumber\\
    \Leftrightarrow 0 =& \quad x_2^2 (x_1^2 - x_2^2 + 1) \\
    &+\frac 1b 4x_2^2 (1-x_1^2) \nonumber\\
    &+\frac 1{b^2} \left( (x_1^2+1)(x_1^2+x_2^2+1) - 2 (-x_1^2+1)^2\right)\nonumber.
\end{align}
This is a fourth degree polynomial depending on $b$, but to understand the corresponding algebraic curves it suffices to understand a quadratic polynomial. To this end, we use two mechanisms: 
\begin{itemize}
    \item Since $b$ is large, we interpret the curve as successive perturbation of the $b$-less term. 
    \item Since all variables appear only as squares, we define the new variables $X=x_1^2, Y=x_2^2$ (axies symmetry along $x_1$ and $x_2$) and study the curve for this new equation. In order to recover the original curve, we then drop everything but the +/+ quadrant and ''unwrap'' it onto the other quadrants.
\end{itemize}

Thus, we look at the curve
\begin{align*}
    0&=Y (X - Y + 1) + \frac 1b 4Y(1-X) + \frac 1{b^2} \left( (X+1)(X+Y+1) - 2 (-X+1)^2\right)\\
    &=Y (X - Y + 1) + \frac 1b 4Y(1-X) + \frac 1{b^2} \left( -X^2+XY+6X +Y-1\right).
\end{align*}

We think of this as the curve $\{Y(X-Y+1) =0\}=\{Y=0\}\cup\{X-Y+1=0\}$, see the left Figure in~\ref{fig:perturbunfold},
first perturbed by 
$\frac 1b 4Y(1-X)$
and then by 
$ \frac 1{b^2} \left( -X^2+XY+6X +Y-1\right).$

\begin{figure}[h]
\centering
\includegraphics[width=.5\textwidth]{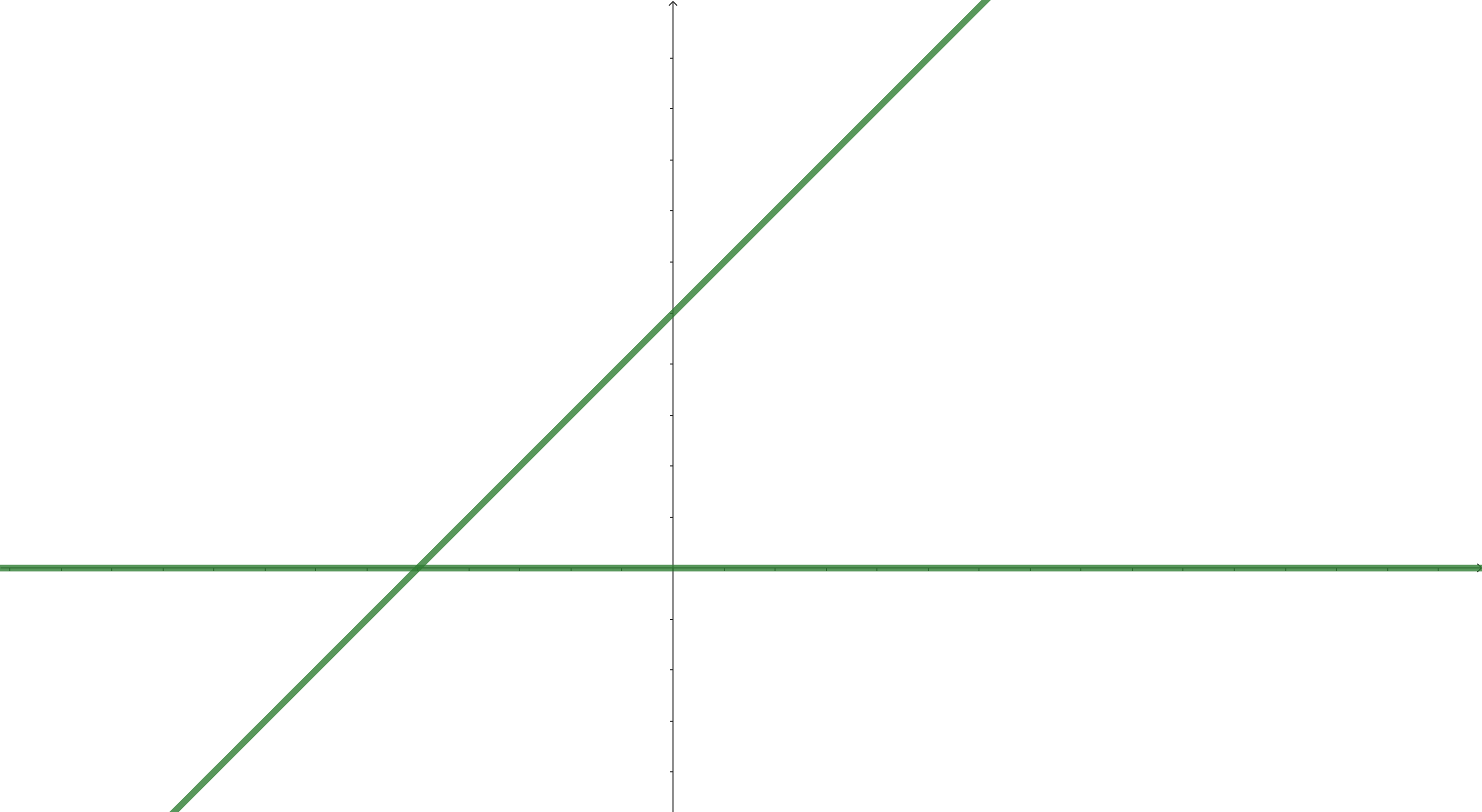}\hfill
\includegraphics[width=.5\textwidth]{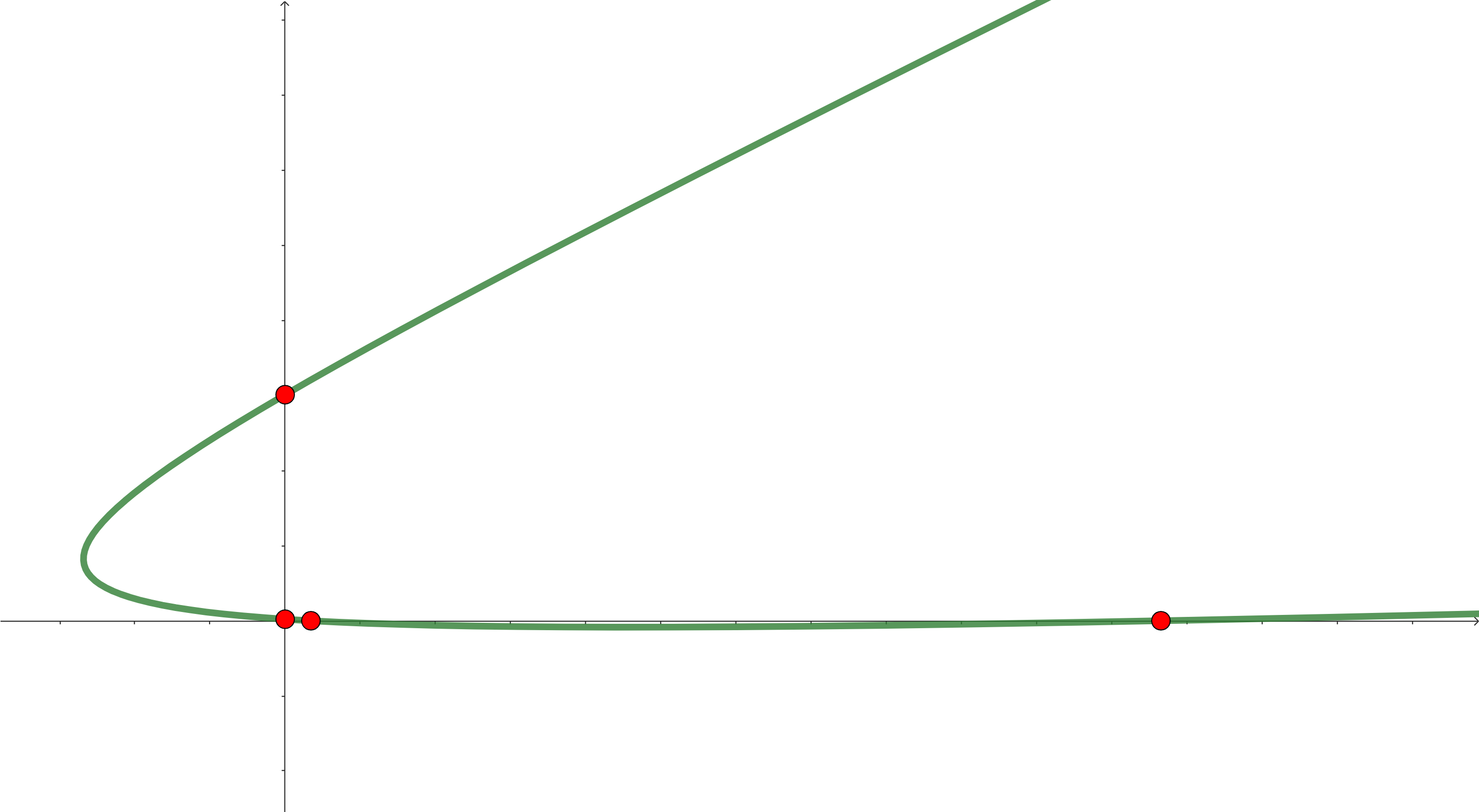}
    \caption{Left one sees the first approximation of the curve for $X$ and $Y$ (for $b=8$), right the second approximation, where the red dots mark the crossing of the axis.}
    \label{fig:perturbunfold}
\end{figure}

The crucial observation is that the curve $(X-Y+1)=0$ is very stable in the positive quadrant and no perturbation will do more than wiggle it. The curve $Y=0$ is very unstable. The first perturbation $4Y(1-X)$ is transversely zero at $Y=0$ and thus does not perturb the curve $Y=0$ after all. The second perturbation, see the middle Figure in~\ref{fig:perturbunfold}, is at $Y=0$ equal to
$$\frac 1{b^2} \left( -X^2+6X -1\right),$$
which is positive in the interval $X\in (3- \sqrt 8,3+\sqrt8)$ and negative outside this interval. Thus, (since $Y(X-Y+1)$ is positive above the $X$-axis), the main line is perturbed down between these two points and up outside these two points. 

Then, by restricting to the positive quadrant and taking the square root of each coordinate, we unfold the quadrant fourfold onto the plane as depicted in the left Figure in~\ref{fig:unfold}. Note that $\pm\sqrt{3\pm \sqrt 8}=\pm(\sqrt 2\pm 1)$. The right Figure in~\ref{fig:unfold} gives an impression of the cone structure.

\begin{figure}[h]
\centering
\includegraphics[width=.5\textwidth]{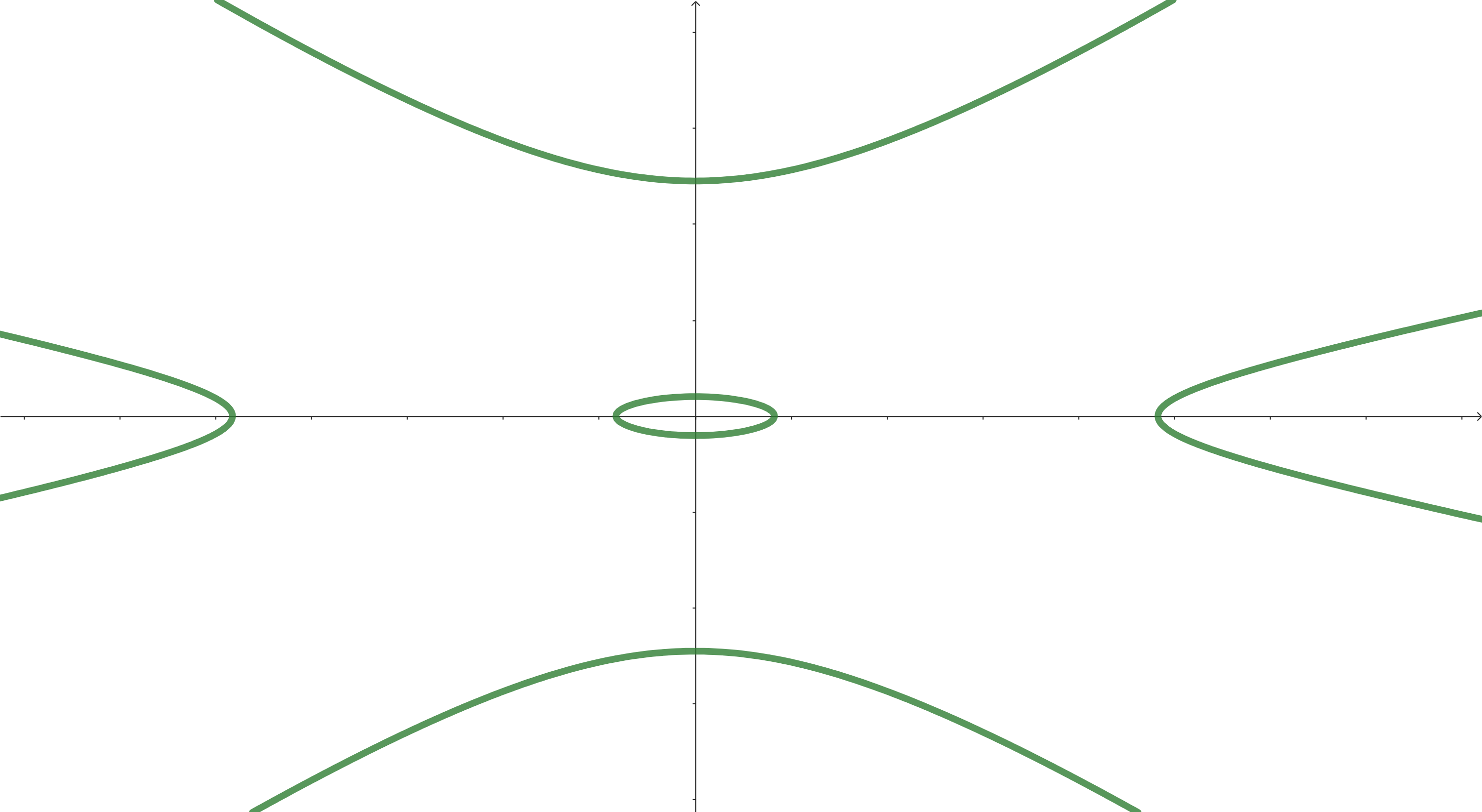}\hfill
\includegraphics[width=.5\textwidth]{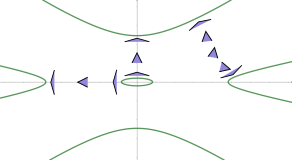}
    \caption{Left one sees the result of taking the square root of the ++ quadrant (`unfolding'). Right, there is a qualitative impression of how the cone fields look like: Close to $\partial \mathcal C$, the cones are close to half-spaces. Between the components of $\Omega^-$, the cones turn (and for $b$ large, they turn close to the  $x_1$-axis).}
    \label{fig:unfold}
\end{figure}

We can clearly see that (for $b=8$, and the behaviour stays the same for larger $b$) the plane is cut by this curve into 6 connected components. Since we know that the connected components of $\Omega^\pm$ and $\Ccal$ are separated by this curve, it suffices to identify the nature of just one vector in each component to find out how to label them. The projectivization of the two dimensional situation in Subsection~\ref{subsec:TwoDim} helps since it describes the circle at infinity, i.e., the situation when $y=0$: The four components that contain the asymptotic ends of the $x_1$- and $x_2$ axis belong to $\Omega$ since at the points $(\pm 1, 0,0)$ and $(0,\pm 1,0)$ the gradient $\nabla f$ is radial and thus $\rho\in C$. More precisely, these four regions  belong to $\Omega^-$ since $\rho$ is outwards pointing and $\nabla f$ is inwards pointing. Similarly, the eye in the middle belongs to $\Omega^+$ since at the point $(0,0,1)$ the gradient $\nabla f$ and $\rho$ coincide. We thus obtain:
\begin{itemize}
    \item $\Omega^-$ consists of the four unbounded simply connected regions of $\partial\Omega^c$,
    \item $\Omega^+$ consists of the bounded region of $\partial\Omega^c$,
    \item $\Ccal$ is the unbounded not simply connected region of $\partial\Omega^c$.
    \item The unbounded components of $\partial\Omega$ are negative barriers, while the bounded component of $\partial\Omega$ is a positive barrier.
\end{itemize}

\bparagraph{Apices along the $x_1$-axis}
As the calculations above have shown, the apices on the $x_1$-axis of the oval and the two curves left and right lie at the points $(\pm(\sqrt 2\pm 1),0)$. Note that this means that there are barrier vectors on the intervals 
$$(-(\sqrt 2 + 1),-(\sqrt 2 -1))\times\{0\}\quad \mbox{ and }\quad  (\sqrt 2 - 1,\sqrt 2 + 1) \times\{0\}.$$
At these points the vertical vector is by symmetry a barrier vector. At the boundaries of the intervals, the vertical vector is the only barrier vector (and lightlike).

%The following geometric consideration that recreates those values may be of interest to build intuition: The point on the $x$-axis with angle $\alpha$ to the $x$-axis is also the point where $\nabla f$ has the angle $\alpha$ to the $x$-axis from the other side. Then, at this point a vector $v$ spans a spacelike barrier if both $2\alpha + \pi/4 < \pi$ and simultaneously $2\alpha > \pi/4$ (the equality cases are when the cone touches the radial line to the right and left). This is an interval $\alpha \in(\pi/8,3\pi/8)$. But then, $\alpha = \arctan x$ which yields the boundaries $x=\pm\sqrt{3\pm \sqrt 8}$ as expected. 

\bparagraph{Apices along the $x_2$-axis} If $x_1=0$, then Equation~\ref{eq:RhoOnCone} becomes
$$b^2 x_2^4 - (b^2+4b+1) x_2^2 +1 = 0,$$
which has zeroes at
$$\pm\frac {\sqrt{b^2+4b+1\pm\sqrt{(b^2+4b+1)^2-4b^2}}}{\sqrt 2 b}.$$ 

It shall be convenient to give a name to the $+-$ version
$$\underline \beta(b)=\underline\beta = \frac {\sqrt{b^2+4b+1-\sqrt{b^4+8b^3+14b^2+8b+1}}}{\sqrt 2 b},$$
and similarly the $++$ version shall be called $\overline \beta$.
For more precise asymptotics of $\underline\beta$, we recall the Taylor expansion of $\sqrt{1+x}=1+\frac12 x -\frac 18 x^2 +O(x^3)$. Then, we can determine the asymptotic behaviour by
\begin{align*}
    \underline \beta^2 &= \frac {b^2+4b+1-\sqrt{b^4+8b^3+14b^2+8b+1}}{ 2 b^2}\\
    &=\frac 12 + \frac 2b + \frac 1{2b^2} -\frac 12 \sqrt{1+\frac8b +\frac{14}{b^2}+\frac 8{b^3}+\frac 1{b^4}}\\
    &= \frac 12 + \frac 2b + \frac 1{2b^2} - \\
    &\quad -\frac 12 \left( 1 + \frac12\left(\frac8b +\frac{14}{b^2}+\frac 8{b^3}+\frac 1{b^4}\right) -\frac 18 \left(\frac8b +\frac{14}{b^2}+\frac 8{b^3}+\frac 1{b^4})\right)^2+O(b^{-3})\right)\\
    &=\frac 12 + \frac 2b + \frac 1{2b^2} - \left( \frac12  + \left(\frac2b +\frac{14}{4b^2}\right) - \left(\frac4{b^2}\right)+O(b^{-3})\right)\\
    &= \frac 1{b^2} + O(b^{-3}),
\end{align*}
and thus we conclude that $\underline \beta(b)\sim \frac 1b$ for $b\to\infty$. Similarly, $\overline\beta(b) \to 1$.

\subsection{Points of causal discontinuity}\label{subsec:PoinsOfDiscontinuity}

As we have seen above, the two following points lie on $\partial\Omega$:
$$q=(\sqrt 2 + 1,0)\in\partial\Omega^-,\quad p=(-(\sqrt 2 - 1),0)\in\partial\Omega^+.$$
By abusing notation we use the same names for the preimages 
$$q=(\sqrt 2 + 1,0,1,0)\in (\pi\circ \varrho)^{-1}(q),\quad p=(-(\sqrt 2-1),0,1,0)\in(\pi\circ\varrho)^{-1}(p).$$ 
We claim that $p$ and $q$ violate causal continuity. We prove this claim by showing that reflectivity does not hold. We split the statement into the following two lemmas:

\begin{lemma}\label{lem:IplusInclusion}
    $I^+(p)\subseteq I^+(q)$.
\end{lemma}
\begin{proof}
    In Lorentz geometry this would follow from the push up lemma. Since our situation is not Lorentz, we perform the proof explicitly. There is a radial lightlike curve from $q$ to 0. From there, we can reach an $\varepsilon$-ring $S_\varepsilon:=(0,0)\times S_\varepsilon^1$ by following radial timelike curves. Finally, from this ring we can timelikely reach almost all of $\Omega_+$. By taking the union over all $\varepsilon$ we have thus found a way to exhaust all of $\Omega_+$ by non-spacelike curves starting at $q$. On the other hand, the future of $p$ is contained in $\Omega^+$. So if we can relax the non-spacelike curves that exhaust $\Omega_+$ to timelike curves, then we are done. The relaxation happens in three steps, see the three parts of Figure~\ref{fig:IplusInc} for reference.
    
\begin{figure}[h]
\centering
\includegraphics[width=.3\textwidth]{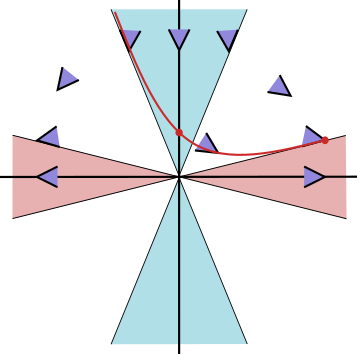}\hfill
\includegraphics[width=.3\textwidth]{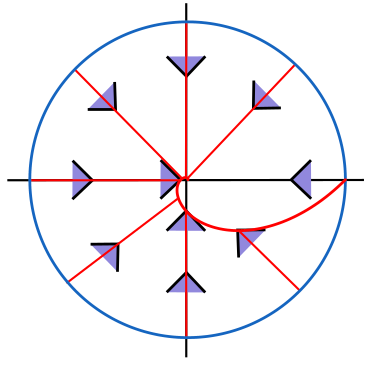}\hfill
\includegraphics[width=.3\textwidth]{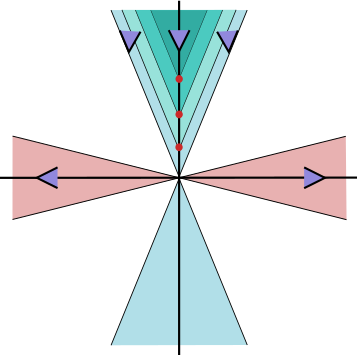}
    \caption{The left figure shows a red hyperbolic timelike trajectory in the $x_1y_1$-plane starting from $q$, missing 0 and intersecting the $y_1$-axis close to 0. The middle figure shows how to reach $S_\varepsilon$ by following a logarithmic spiral in the $y_1y_2$-plane. The right figure illustrates how to fill out $\Omega^+$ from $S_\varepsilon$ with $\varepsilon\to0$.}
    \label{fig:IplusInc}
\end{figure}
    
    We first relax the radial lighlike curve from $q$ to 0 to a timelike curve from $q$ to $(0,0,\varepsilon^2,0)$. For this it suffices to work in the $x_1,y_1$-plane, which we call here $x,y$-plane for ease of notation. Note that in this plane the gradient vector field is $\nabla f (x,y)=\begin{pmatrix}-x\\y\end{pmatrix}= \begin{pmatrix} -1 & 0 \\0&1\end{pmatrix}\begin{pmatrix} x\\y\end{pmatrix}.$ The lightlike direction that radially points towards the origin is then $\nabla f$ rotated counterclockwise by $\frac \pi 4$ and we define $V:=\operatorname{Rot}_{\frac \pi 4} \nabla f$. We verify
    \begin{align*}
        V q &= \frac 1{\sqrt 2}\begin{pmatrix} 1&-1\\1&1\end{pmatrix} \begin{pmatrix} -1 & 0 \\0&1\end{pmatrix}q  = \frac 1{\sqrt 2}\begin{pmatrix} -1&-1\\-1&1\end{pmatrix} \begin{pmatrix} \sqrt 2 + 1\\ 1 \end{pmatrix}\\
        &= \frac 1{\sqrt 2} \begin{pmatrix}-(\sqrt 2 + 1) -1 \\- (\sqrt 2 + 1) +1 \end{pmatrix} = \frac 1{\sqrt 2} \begin{pmatrix}- \sqrt 2(\sqrt 2 + 1)\\ - \sqrt 2\end{pmatrix}
        \\
        &= -q,
    \end{align*}
    where the scaling chosen here implies that the integral curve exponentially converges to 0: $\varphi_V^t q = e^{-t} q$. The relaxation is performed by pushing off the lightlike vector field by rotating it clockwise by an angle $\epsilon\ll\varepsilon^2$, which makes it into a timelike vector:
    \begin{align*}        
    W &:= \operatorname{Rot}_{\frac \pi 4-\epsilon} \nabla f = \frac 1{\sqrt 2}\begin{pmatrix} -\cos\epsilon +\sin\epsilon &-\cos\epsilon -\sin\epsilon \\-\cos\epsilon -\sin\epsilon & \cos\epsilon -\sin\epsilon \end{pmatrix} \\
    &= V+ \frac \epsilon{\sqrt{2}}\begin{pmatrix} 1 & -1\\-1 &-1\end{pmatrix} + O(\epsilon^2) = V+ \epsilon \operatorname{Rot}_{-\frac\pi 2} V + O(\epsilon^2). 
    \end{align*}
    Thus, $W$ is a perturbation of $V$. Since the Eigenvalues of $V$ are $\pm 1$ (and thus far from each other), the Eigenvalues of $W$ will be close to $\pm 1$ and thus the matrix is diagonalizable with Eigenvectors $\tilde q_\pm$ close to the Eigenvectors of $V$, which are $q=(1+\sqrt 2, 1)$ and $(1-\sqrt 2,1)$. The push-off direction makes it clear that $q$ lies slightly above $\tilde q_+$, so the integral curve $\varphi^t_Wq$ starting at $q$ follows a tight hyperbola whose forward and backward asymptotics are the two rays $\RR_+ \tilde q_\pm$. In particular, this curve crosses the positive $y$-axis at a point that lies between 0 and $(0,\varepsilon^2)$ since we chose $\epsilon$ small in comparison to $\varepsilon^2$. That is where we stop the trajectory and go radially out until we reach $(0,\varepsilon^2)$.

    \vspace{1em}
    We now show that from $(0,0,\varepsilon^2,0)$ we can still timelikely connect to any point on the ring $\varepsilon(0,0,\cos\sigma,\sin\sigma)\in S_\varepsilon$. This is easily done by first following a logarithmic spiral outwards in the plane $\mathbf x=0$ until the angle is correct, then we go radially out until we reach $S_\varepsilon$. Note that this is possible since any logarithmic spiral with angle to the radial curve less than $\pi/2$ is timelike, and since the radial ratio between two consecutive branches of a logarithmic spiral has a constant factor the spiral can for $\varepsilon$ small enough perform a full turn before growing in radius from $\varepsilon^2$ to $\varepsilon$. Having connected $q$ to every point on the ring $S_\varepsilon$, we can now timelikely exhaust $\Omega_+$ and the proof is done. 

\end{proof}

\begin{remark}
    In fact, this proof works (by careful choice of push-off of the initial lightray from $q$ to $0$) for any point $q$ in the preimage of $\partial\Omega^-$ and any point $p$ in the preimage of $\partial\Omega^+$.  
\end{remark}

\begin{lemma}\label{lem:IminusExclusion}
    There is $b_0$ such that for all $b>b_0$ we have $I^-(q)\nsubseteq I^-(p)$. 
\end{lemma}

We will even show that the intersection $I^-(q)\cap I^-(p)$ is empty. For this, we construct a non-timelike barrier going through $p$ such that $I^-(p)$ lies completely on the left of the barrier, but $I^-(q)$ is contained in the rightmost island of $\Omega^-$, which will lie on the right hand side of the barrier. We describe the upper half of the barrier (for $x_2>0)$ piecewise, the lower half will be a reflection of the upper half, see Figure~\ref{fig:barrier}. 
\begin{itemize}
    \item The first piece of barrier starts at $p=(-(\sqrt 2 -1),0)$ and follows $\partial\Omega^+$ until $x_1=-0.1$. This segment is a barrier by the construction of $\Omega$. 
    \item An interpolation barrier deviates from $\partial\Omega^+$ essentially in a quadratic function described in Lemma~\ref{lem:BreachTheGap}. It will arrive at a point $(0,(1+\varepsilon)/b)$, which is the starting point of the third barrier.
    \item The third piece of barrier is a hyperbola described in Lemma~\ref{lem:HyperbolicBarrier}
\end{itemize}

\begin{figure}[h]
\centering
\includegraphics[width=.5\textwidth]{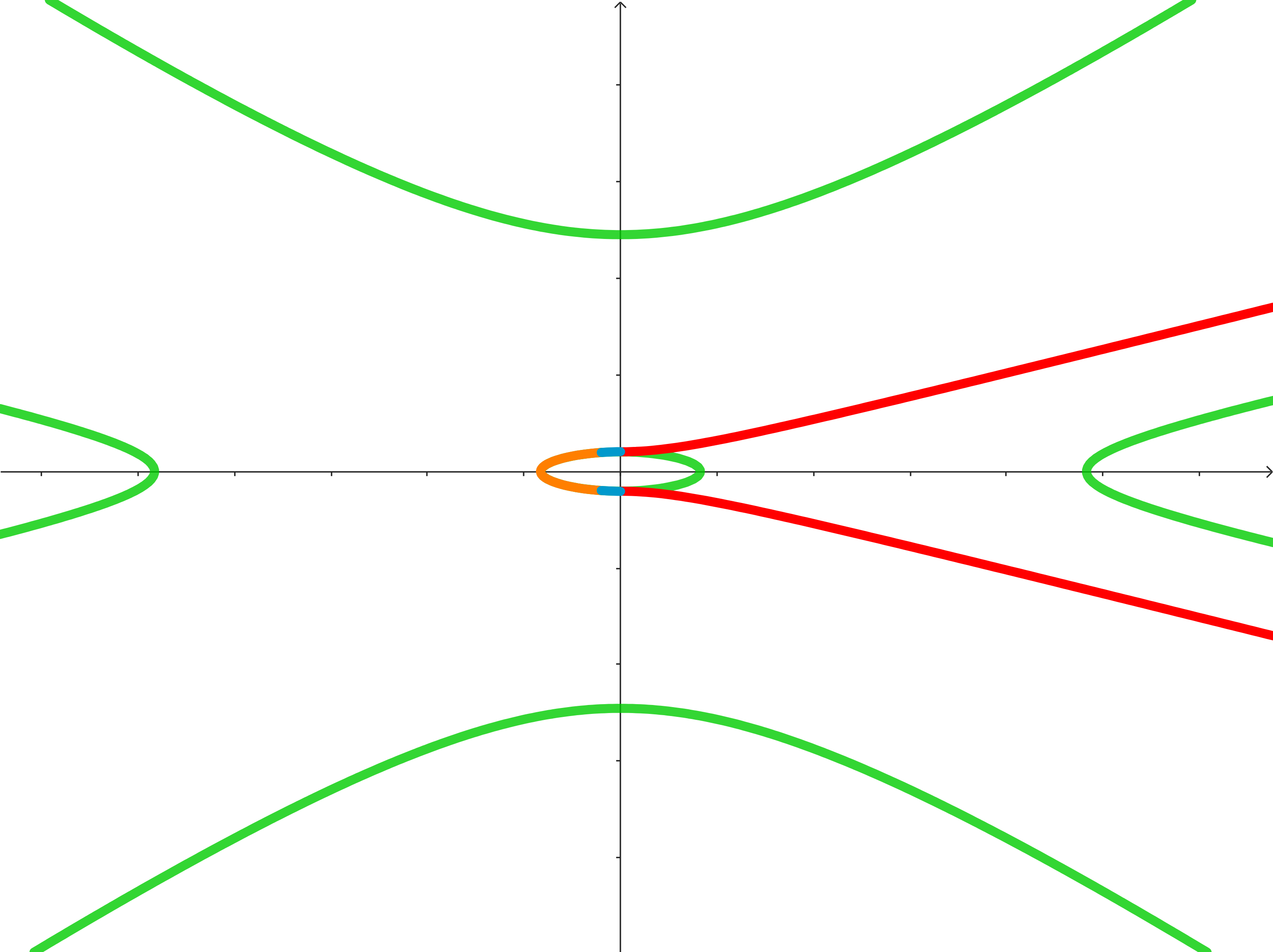}\hfill
\includegraphics[width=.5\textwidth]{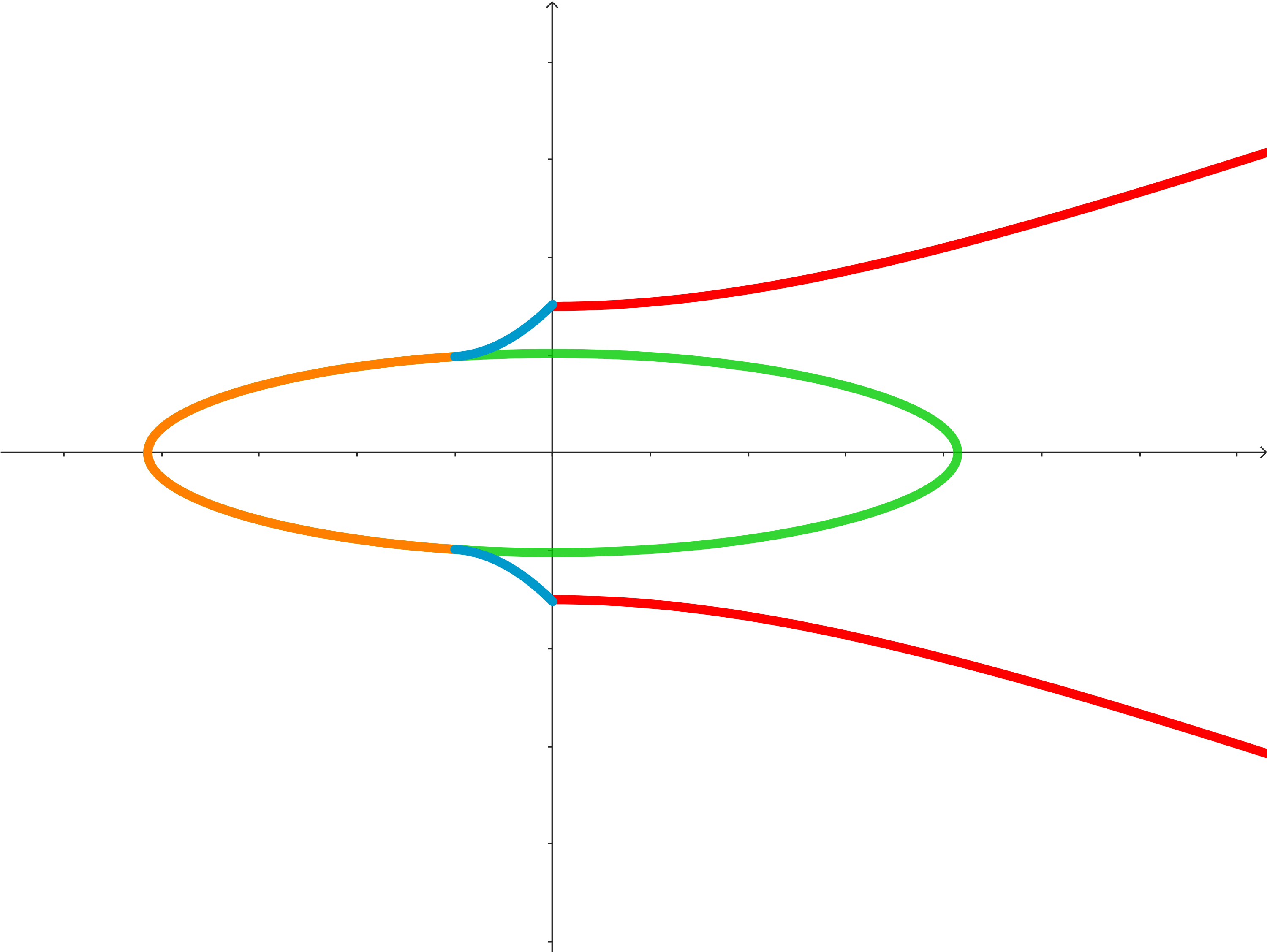}\hfill
    \caption{On the left is the actual barrier from Lemma 3.9 (with the parameters $\beta=0.102,a=6$ found in Section~\ref{sec:B8} for $b=8$. The orange curve is the first piece following $\partial\Omega$, the blue line is the interpolation curve and the red line is the hyperbolic piece of barrier. On the right side is an exaggerated and zoomed in version to see better the quality of the curves.}
    \label{fig:barrier}
\end{figure}

\begin{remark}
    Our barrier will show that Lemma~\ref{lem:IminusExclusion} is also true for a preimage of $\partial\Omega^-$ in the same component as $q$ and a preimage of $\partial\Omega^+$ nearby $p$. Thus, there are many points of causal discontinuity on the boundary of past/future of the critical point. 
\end{remark}

\subsection{Characterization of barriers}\label{subsec:horizontal}
We search not only barrier vectors, but a barrier. This is a 1-dimensional submanifold whose tangent vectors all admit the above relation. Let $\gamma(t)=(x_1(t),x_2(t))$ parametrize this submanifold. Then, $(v_1,v_2)=(\dot x_1,\dot x_2)$. Since we are free to reparametrize, we may choose $v_1=1$, so that $x_1=t$. Then (with writing $x=x_2, v=v_2, \|\nabla f\|^2=G=t^2+b^2x^2+1$), the condition for non-timelikeness from Lemma~\ref{lem:CriterionC2}  becomes
\vspace{1em}
\begin{equation}\label{eq:Criterion3}
    \begin{split}
    2(-2v t + (b+1)x)^2 \geq& {(1 +v^2 + (x-vt)^2)G},\\
   \Leftrightarrow \quad 0 \leq& v^2(8t^2 -(1+t^2)G)\\
    &+2vxt(-4(b+1)+G)\\
    &+(2(b+1)^2x^2-(1+x^2)G),\\
   \mbox{We rename the coefficients: }  0 \leq& Av^2+ 2B v+C.
   \end{split}
\end{equation}
Using the midnight formula, the vectors that provide a non-timelike barrier are thus $v\in [v_-,v_+]$ with
$$v_\pm = \frac{-B}A \pm \sqrt{\frac\Delta {A^2}}$$
if the discriminant is non-negative 
$$\Delta=B^2-AC\geq 0,$$
and otherwise there are no non-timelike barrier vectors. 
We compute (note that the terms that do not contain $G$ cancel out)
\begin{align*}
    \Delta &= x^2t^2(G-4(b+1))^2-(8t^2 -(1+t^2)G)(2(b+1)^2x^2-(1+x^2)G)\\
    &= G(-t^4+(b^2-4b+1)t^2x^2-b^2x^4 + 6t^2+(b^2+4b+1)x^2-1)\\
    &= -G( x^4b^2 - x^2((b^2-4b+1)(1+t^2)+ 8b) + (t^4 -6t^2 + 1) )
\end{align*}  
This expression is non-negative whenever the second factor is non-positive and then there exists a barrier. For the equality (lightlike) case, we can use the midnight formula for $x^2$:
\begin{align*}
    x_{\pm\pm}(t) & = \pm \frac 1{\sqrt 2 b} \sqrt{(b^2-4b+1)(1+t^2)+ 8b\pm\ldots}\\
    & \quad \sqrt{\ldots \pm \sqrt{((b^2-4b+1)(1+t^2)+ 8b )^2 -4 b^2(t^4 -6t^2 + 1)}}.
\end{align*}
We are, of course, interested in the $(+-)$-choice of signs. We start the investigation of the function $x_{+-}$ by proving that it is monotone falling for $t>0$ small. This is the same as saying that $2b^2 x_{+-}^2$ is monotone falling for $t>0$ small. But for this note that the function has the form $\varphi(t)-\sqrt{\varphi^2(t)-\psi(t)}$, which has the derivative 
$$\varphi'(t)-\frac{\varphi(t)\varphi'(t)}{\sqrt{\varphi^2(t)-\psi(t)}} + \frac{\psi'(t)}{2\sqrt{\varphi^2(t)-\psi(t)}}.$$
Since for $t>0$ small we have $\varphi>0,\varphi'>0,\psi>0,\psi'<0$, the derivative is indeed negative. By symmetry, the function is monotone rising for $t<0$. 

To gauge the magnitude of the function, we estimate by completing the square $$\varphi^2-\psi \leq \varphi^2 - \psi + (\frac1{2}\frac \psi\varphi)^2=(\varphi-\frac1{2}\frac\psi\varphi)^2$$
and thus 
\begin{align}\label{eq:BoxEstimate}
    x_{+-}&\geq\frac1{\sqrt 2 b}\sqrt{ \frac \psi{ 2\varphi}}=\frac 1{2b}\sqrt{\frac \psi\varphi}=\sqrt{\frac{1-6t^2+t^4}{(b^2-4b+1)(1+t^2)+8b}}.
\end{align}
At $t=0$ this evaluates to 
$$x_{+-,max} \geq \frac 1{2b}\sqrt{\frac{4b^2}{b^2+4b+1}}=\frac{1}{\sqrt{b^2+4b+1}},$$
which means that asymptotically it is essentially $\frac1b$ (which reconfirms the earlier findings).
For $t=-0.1$ we estimate  
\begin{align*}
    x_{+-}(-0.1) &\geq 0.964776 \frac 1b + O(\frac1{b^2}).
\end{align*}
Asymptotically for $b\to\infty$, this estimate will approximate $0.96\frac1b$. For the estimates below it will suffice to assume $x_{+-}(t)\in\frac1b(\frac 34,\frac 32)$ for all $t\in[-0.1,0]$.

\subsection{A hyperbolic barrier}

Here, we describe the hyperbolic barrier starting at $(0,\beta)$ for $\beta\in \pm (\underline\beta,\overline\beta)$. We are interested in $\beta$ close to $\underline\beta\sim\frac 1b$, so we choose $\beta=(1+\varepsilon)/b$.
\begin{lemma}\label{lem:HyperbolicBarrier}
    For any $\varepsilon>0$ small enough there is $b$ large enough such that for any $a$ and for $\beta=(1+\varepsilon)/b$ the curve 
    $$(t, \gamma(t)),\quad t\in[0,\infty), \quad \gamma(t) = \beta \sqrt{1+a t^2}$$
    is a spacelike barrier.
\end{lemma}
To prove the lemma, we investigate for which $b$ the derivative $(1,v(t))$ is a barrier vector (for all $t$). Note that this is just the piece in the $++$ quadrant, for the lower branch add a -. Also note that the asymptotic slope $\sqrt a\beta$, which we hope to be a timelike ray.

\begin{proof}
By our choices we have $\gamma=\beta\sqrt{1+at^2}$ and $v=a\beta t\frac 1{\sqrt{1+at^2}}$ and thus $v\gamma = a\beta^2 t$. The normal vector then is $n= (1,v,0) \times (t,\gamma, 1) = (v,-1,\gamma-vt)$ and we can compute the quantities in the criterion for being a barrier 
\begin{align*}    
    \langle n, \nabla f(t,\gamma,1)\rangle^2 &= (-vt + b\gamma + \gamma -vt )^2= ((b+1)\gamma -2vt)^2\\
    &= (b+1)^2\gamma^2 - 4(b+1) v\gamma t + 4v^2t^2 \\
    &=(b+1)^2\beta^2(1+at^2) - 4(b+1) a\beta^2 t^2 + \frac {4 a^2\beta^2 t^4}{1+at^2}\\
    &= \frac {4 a^2\beta^2 t^4}{1+at^2} + ( b^2-2b-3 ) a\beta^2 t^2 + (b+1)^2\beta^2, \\    
    |n|^2&=v^2+1+(\gamma -vt)^2 = v^2(1+t^2)-2vt\gamma + (1+\gamma^2),\\
    &= \frac {a^2\beta^2t^2(1+t^2)}{1+at^2} - 2a\beta^2 t^2 + 1 + \beta^2 (1+at^2) \\
    &= \frac {a^2\beta^2t^2(1+t^2)}{1+at^2} +(-a\beta^2)t^2 + (1 + \beta^2), \\
    |\nabla f|^2 &= 1+t^2+b^2\gamma^2 = (1+b^2\beta^2) + (1+ab^2\beta^2)t^2
\end{align*}
Thus, the criterion for being a barrier reads as usual
$$\left(\frac{\langle n,\nabla f\rangle}{|n||\nabla f|}\right)^2  \geq\frac 12 \quad \Leftrightarrow \quad 2\langle n,\nabla f\rangle^2 -|n|^2|\nabla f|^2 \geq 0$$
which we now multiply out and sort by powers of $t$.
\begin{align*}
    0\leq& 2\frac {4 a^2\beta^2 t^4}{1+at^2} + 2( b^2-2b-3 ) a\beta^2 t^2 + 2(b+1)^2\beta^2 \nonumber\\
    &- \left(\frac {a^2\beta^2t^2(1+t^2)}{1+at^2} +(-a\beta^2)t^2 + (1 + \beta^2)\right)((1+b^2\beta^2) + (1+ab^2\beta^2)t^2)\nonumber\\
    =& \frac {8 a^2\beta^2 t^4-a^2\beta^2t^2(1+t^2)((1+b^2\beta^2) + (1+ab^2\beta^2)t^2)}{1+at^2}\nonumber\\
    &+  2( b^2-2b-3 ) a\beta^2 t^2 + 2(b+1)^2\beta^2 \nonumber\\
    &+\left(a\beta^2t^2 - (1 + \beta^2)\right)((1+ab^2\beta^2)t^2+(1+b^2\beta^2))\\       
    =& \frac1{1+at^2} (-a^2\beta^2(1+ab^2\beta^2)t^6 + (8 a^2\beta^2 - a^2\beta^2(1+b^2\beta^2)\nonumber \\ 
    &\quad - a^2\beta^2 (1+ab^2\beta^2))t^4  -a^2\beta^2(1+b^2\beta^2)t^2)\nonumber\\
    &+ t^4 a\beta^2(1+ab^2\beta^2)  \nonumber\\
    &+t^2(2( b^2-2b-3 ) a\beta^2 +a\beta^2(1+b^2\beta^2) - (1 + \beta^2)(1+ab^2\beta^2))\nonumber\\
    &+ 2(b+1)^2\beta^2 -  (1 + \beta^2)(1+b^2\beta^2)\nonumber\\
     =& \frac { -a^2\beta^2(1+ab^2\beta^2)t^6 + (6 a^2\beta^2 - a^2b^2\beta^4- a^3b^2\beta^4 )t^4  -a^2\beta^2(1+b^2\beta^2)t^2}{1+at^2}\\
    &+t^4a\beta^2(1+ab^2\beta^2) \nonumber \\
    &+t^2(( b^2-4b-5 ) a\beta^2 - 1 - \beta^2 )\nonumber\\
    &+(- b^2\beta^4 + (b^2+4b+1)\beta^2 - 1)\nonumber.
\end{align*}
We multiply by $1+at^2$ and we obtain
\begin{align*}    
    0\leq& (-a^2\beta^2(1+ab^2\beta^2)t^6 + (6 a^2\beta^2 - a^2b^2\beta^4- a^3b^2\beta^4 )t^4  -a^2\beta^2(1+b^2\beta^2)t^2)\nonumber\\
    &+(1+at^2)t^4  a\beta^2(1+ab^2\beta^2)  \nonumber\\
    &+(1+at^2)t^2(( b^2-4b-5 ) a\beta^2 - 1 - \beta^2 ) \nonumber\\
    &+(1+at^2)(- b^2\beta^4 + (b^2+4b+1)\beta^2 - 1)\nonumber\\
    =& \quad\; t^6(-a^2\beta^2(1+ab^2\beta^2)+ a(a\beta^2(1+ab^2\beta^2))) \nonumber\\
    & + t^4((6 a^2\beta^2 - a^2b^2\beta^4- a^3b^2\beta^4) +a\beta^2(1+ab^2\beta^2) \\ 
    &\quad +a(( b^2-4b-5 ) a\beta^2 - 1 - \beta^2 ) ) \nonumber\\
    & + t^2(-a^2\beta^2(1+b^2\beta^2) + (( b^2-4b-5 ) a\beta^2 - 1 - \beta^2 ) \\
    &\quad + a(- b^2\beta^4 + (b^2+4b+1)\beta^2 - 1))\nonumber \\
    &+(- b^2\beta^4 + (b^2+4b+1)\beta^2 - 1),
\end{align*}
so by simplifying we finally find the form
\begin{equation}\label{eq:ConfirmHyperbolicBarrier}
    \begin{split}
    0\leq& \quad\; t^4( a^2\beta^2(b^2-4b+1) - a^3b^2\beta^4 - a ) \\
    & + t^2(- (a^2 + a)b^2\beta^4 +(2ab^2 - a^2- 4a - 1)\beta^2 - (1 +a) ) \\
    &+(- b^2\beta^4 + (b^2+4b+1)\beta^2 - 1).
    \end{split}
\end{equation}

\bparagraph{Asymptotic analysis}
We fix $a>1$ and $\varepsilon>0$ arbitrarily. If we set $\beta=(1+\varepsilon)/b,$ then for $b\to \infty$ the above inequality becomes
\begin{align*}
0\leq& t^4( a^2(1+\varepsilon)^2 - a +O(b^{-1})) \\
    &t^2( +2a(1+\varepsilon)^2 - (1+a) +O(b^{-1})) ) \\
    &+( (1+\varepsilon)^2 - 1 +O(b^{-1}))).
\end{align*}
Note that all coefficients are asymptotically positive  and thus the inequality is asymptotically satisfied. This is interesting for small $\varepsilon$ since it is then easier to connect to the starting point of the hyperbolic barrier. 

Please note that for a fixed $b$ the terms of lower order do matter.
\end{proof}

%-----------------------------------------------------------
%-----------------------------------------------------------
%-----------------------------------------------------------
%-----------------------------------------------------------
%-----------------------------------------------------------
%-----------------------------------------------------------

\subsection{The interpolation curve}\label{subsec:interpolation}
We will describe how we deviate from $\partial\Omega^+$ to the starting point of the hyperbola above. 

\begin{lemma}\label{lem:BreachTheGap}
    For $b$ large enough, there is a barrier starting at $(-0.1,x_{+-}(-0.1))$ and ending at $(0,(1+\varepsilon)/b)$.
\end{lemma}

The deviation process is similar to Example 1.11 in~\cite{CG}. We explain here the low regularity situation and in Lemma~\ref{lem:FloatAway} we recreate the example for low regularity causal structures that are general enough to apply to our situation.

\bparagraph{Hölder continuity of the causal structure}
The inner region $\Omega^+$ is roughly an ellipse, and its boundary is a non-timelike barrier (by construction). 
If we take a section of vector fields $V$ in $\Omega^c$ such that $V$ is 'between' $v_+$ and $v_-$, then an integral curve of $V$ is also a non-spacelike boundary. At the boundary of $\Omega_+$, $V$ follows $T\partial\Omega_+$, since there $v_\pm$ coincide and $V$ has no other choice. In the representation above, $V$ then has the form $V=(1,v)$  with
\begin{equation}\label{eq:HoelderVectorField} v = \frac{-B-D}A = \frac{-xt(G-2(b+1)) - D}{8t^2-(1+t^2)G} = \frac{xt(2b+1-b^2x^2-t^2)+D}{t^4+t^2x^2b^2+x^2b^2-6t^2+1}
\end{equation}
where $D\in[-\sqrt\Delta,\sqrt\Delta]$ is arbitrary.

\vspace{1em}
We start at the the boundary of $\Omega^+$, so the integral curve of $V$ has no choice but to be tangent to $\partial\Omega^+$. But the entire boundary $\partial\Omega_+$ is an integral curve of $V$ with those initial conditions. We want, however, that the integral curve at some point separates from $\Omega^+$. Thus, we can find such a barrier only if the integral curves of $V$ are not uniquely determined by initial condition. Luckily, $V$ is not Lipschitz but rather $1/2-$Hölder for a reasonable choice of $D$, assuming that $\Delta$ is transverse to zero at $\partial\Omega$. The magnitude of transversality will be determined in the proof of Lemma~\ref{lem:BreachTheGap}. Lemma~\ref{lem:FloatAway} below shows that \emph{given} transversality, there is a drift. As a consequence, we see that between the past $I^-(p)$ and $\partial\Omega^+\subseteq J^-(p)$ there is a causal bubble in the sense of~\cite{CG}. 

\begin{lemma}\label{lem:FloatAway}
    Assume $V:[0,\infty)\times\RR\rightarrow\RR$ satisfies 
    \begin{enumerate}
        \item  $V$ is Lipschitz for $x>0$,

        \item $V(0,t)=0$ for any $t\in\RR$,
        
        \item $c\sqrt x \leq V(x,t)\leq C\sqrt x$ for some uniform constants $c,C>0$.
    \end{enumerate}
     Then, for any $T>0$ there is a solution $X:\RR\rightarrow[0,\infty)$ to 
     \begin{equation}\label{ODE}
     X'(t)=V(X(t),t)  \end{equation}
     such that $X(t)=0$ for $t\leq T$ and 
    \[
    X(t)\geq\frac14c^2(t-T)^2 \quad\text{for all}\quad t\geq T.
    \]
    \end{lemma}

\begin{proof}
Through any initial data $(x_0,t_0)$, the integral curve $(X(t),t)$ to the vector field $(V(x,t),1)$ exists and is unique as long as $X(t)>0$ by the theorem of Lindelöf-Picard. We show that there exists $\tau(x_0,t_0)\in[\tau_c(x_0,t_0),\tau_C(x_0,t_0)]$ and a unique solution $X:\RR\rightarrow[0,\infty)$ to \eqref{ODE} such that $X(t_0)=x_0$ and 
\begin{equation}
\begin{split}
X(t)&>0\,\,\text{for}\,\,t\in(\tau,\infty),\\
X(t)&=0\,\,\text{for}\,\, t\in(-\infty,\tau).
\end{split}
\end{equation}
Notice that $(3)$ implies that 
\begin{equation}\label{comp}
\begin{split}
\frac{C^2}{4}(t-\tau_C(x_0,t_0))^2&\leq X(t)\leq\frac{c^2}{4}(t-\tau_c(x_0,t_0))^2,\quad \text{for}\,\,\, t\in[\tau_C(x_0,t_0),t_0],\\
\frac{c^2}{4}(t-\tau_c(x_0,t_0))^2&\leq X(t)\leq\frac{C^2}{4}(t-\tau_C(x_0,t_0))^2,\quad \text{for}\,\,\, t\in[t_0,\infty),
\end{split}
\end{equation}
where $\tau_a(x,t) = t -\frac 2a \sqrt {x}$. Thus, $X(t)>0$ for $t\in[t_0,\infty)$. We set, for $n\in\mathbb{N}$, 
$$
(x_{n+1},t_{n+1})=(X(\tau_C(x_n,t_n)),\tau_C(x_n,t_n))
$$ 
and notice that for $t\in[t_{n+1},t_n]$, \eqref{comp} holds with every subscript $0$ replaced by $n$. It follows that either $x_{n+1}=0$ and $\tau=\tau_C(x_n,t_n)$ or 
\begin{equation}\label{iter}
\begin{split}
&0<x_{n+1}<\bigg(\frac{C-c}{C}\bigg)^2 x_n,\\
&[\tau_c(x_{n+1},t_{n+1}),\tau_C(x_{n+1},t_{n+1})]\subset[\tau_c(x_n,t_n),\tau_C(x_n,t_n)].
\end{split}
\end{equation}
Since $0\leq\tau_C(x,t)-\tau_c(x,t)\leq(\frac{1}{c}-\frac{1}{C})\sqrt x$, \eqref{iter} implies that there is $\tau$ contained in $[\tau_c(x_n,t_n),\tau_C(x_n,t_n)]$ for any $n\in\mathbb{N}$. The continuity of the solution gives 
$$
X(\tau)=\lim_{n\rightarrow\infty}X(t_n)=\lim_{n\rightarrow\infty}x_n=0.
$$
For $t<\tau$ we set $X(t)=0$ and it is easily seen that such an $X$ is unique.

\vspace{1em}
Now choose a sequence of points $(x_n,T)$ with $x_n>0$ that converges to $(0,T)$. By the above construction, the corresponding sequence of solutions of the Cauchy problem $(X_n(t),t)$ have interval of unique existence $(\tau_n,\infty)$ with $\tau_n<T$ and $\tau_n\to T$. Each of the solutions satisfies for $t>T$ the inequality $X_n(t)>\frac14c^2(t-T)^2$. Because of the continuous dependence of the solutions of Cauchy problems on their initial conditions, $(X_n(t),t)|_{[T,\infty)}$ converges to an integral curve of the vector field with the demanded properties.
\end{proof}

\begin{proof}[Proof of Lemma~\ref{lem:BreachTheGap}]
Lemma~\ref{lem:FloatAway} shows that at any point one can drift away from the barrier, essentially with the speed of the differential equation $\dot x = c\sqrt{x}$, which has the solution $\frac 14 c^2(t + const)^2$. If $b$ is large enough, then we will be easily able to drift up by a distance of $\varepsilon b$, which bridges the gap to the hyperbolic piece of barrier.

Remember that in Equation~\ref{eq:HoelderVectorField} the term $\frac DA$ was bounded by the square root of $\frac \Delta{A^2}=\frac{B^2-AC}{A^2}$. We want to apply Lemma~\ref{lem:FloatAway} to this vector field, but for this we require a linear estimate from below of this expression (as we know that along $\partial\Omega$ it is zero).

Thus, we compute the derivative of $\frac{\Delta}{A^2}$. 
$$\partial_x\frac\Delta{A^2} = \partial_x\frac{B^2-AC}{A^2}= \frac{2\partial_x B B A - 2\partial_x A  B^2 - \partial_x C A^2 +\partial_x A A C}{A^3}.$$
We recollect the variables and derivatives, where we use that $t\in[-0.1,0], xb= X\in [\frac 34,\frac 32]$ for $b$ large. 
\begin{align*}
    G&=1+t^2+X^2 \in [1.5,3.26]+O(\frac 1b),\\
    A&=8t^2 -(1+t^2)(1+t^2+X^2) \\
    &= -(1 +X^2 +t^2(X^2-6)+t^4) \in [-3.25,-1.5]+O(\frac 1b) ,\\
    B&=xt(-4(b+1)+G)=-4tX+O(\frac 1b) ,\\
    C&=2(b+1)^2x^2-(1+x^2)G = X^2 - 1-t^2 +O(\frac1b)
\end{align*}
\begin{align*}
    \partial_xA&=-(1+t^2)2bX ,\\
    \partial_xB&=t(-4(b+1)+G)+ xt2b^2x =-4tb+O(1) ,\\
    \partial_xC&=4(b+1)^2x-2xG-(1+x^2)2b^2x =2bX+O(1),\\
    2\partial_x B B A &= 2(-4tb+O(1))(-4tX+O(\frac 1b)) A =   32AXt^2b+O(1) ,\\
    - 2\partial_x A  B^2 &= -2(-(1+t^2)2bX)(-4tX+O(\frac 1b))^2 = 64t^2(1+t^2) bX^3+O(1)  ,\\
    - \partial_x C A^2 &= -(2bX+O(1)) A^2 = -2A^2Xb + O(1),\\
    +\partial_x A A C &= (-(1+t^2)2bX + O(1))A (X^2-1 -t^2 +O(\frac1b)) \\
    &= -(1+t^2)(X^2-1-t^2)2AXb +O(1).
\end{align*}
Thus, we estimate
\begin{align*}
    A^3\partial_x\frac\Delta{A^2} &= b\left(32AXt^2+64t^2(1+t^2) X^3-2A^2X-(1+t^2)(X^2-1-t^2)2AX\right)+O(1)\\
     &= b\left(32(A+2X^2)Xt^2+64t^4 X^3-2A^2X-(1+t^2)(X^2-1-t^2)2AX\right)+O(1)\\
     &\leq b\left(64t^4 X^3-2AX[A-X^2(1-t^2+32t^2\frac XA)+1-17t^2]\right)+O(1)\\
     &\leq b\left(0.022- 2AX[A-0.49X^2+1]\right)+O(1)\\
     &\leq b\left(0.022-2AX[-0.77]\right)+O(1) < -1.5 b + O(1).\\
    \partial_x\frac\Delta{A^2} &\geq \frac {-1.5b}{A^3}\geq 0.04b.
\end{align*}

Loosely speaking, we have shown that it is possible to choose the barrier $x_{+-}(t)+\frac 14 \alpha (t+0.1)^2$ for $\alpha$ in the order of magnitude of $0.04 b$. This way, the barrier reaches at $t=0$ a height of $x_{+-}(0)+0.0004b$, which for large $b$ is enough to breach the gap $(1+\varepsilon)/b - \underline \beta$ (recall that $\underline \beta\sim\frac1b$, so we overshoot our goal by two orders of magnitude).

More formally, we follow the vector field $\frac{-B}A+\frac{-D}A$, where the previous computation shows that we can freely choose $\frac{-D}A \leq \sqrt{0.04b(x-x_{+-})}$. Since $\frac{-B}A\geq 0$, the height of the solution can be estimated from below by ignoring the first summand $$x(t)\geq x_{+-}+0.04b(t+0.1)^2$$
as long as this height does not exceed $\frac 32 \frac 1b$, at which point the estimate above breaks down. However, since $\frac32\frac 1b> (1+\varepsilon)/b$ for small enough $\varepsilon>0$, we do not mind the breakdown of estimates. By Lemma~\ref{lem:FloatAway}, we can then choose a suitable later departure point (or just a smaller slope) to reach the target height of $(1+\varepsilon)/b$ precisely.
\end{proof}

%-----------------------------------------
%-----------------------------------------
%-----------------------------------------
%-----------------------------------------
%-----------------------------------------
%-----------------------------------------
%-----------------------------------------
%-----------------------------------------

\section{Specific barrier for $b=8$}\label{sec:B8}
In this section we show that the assertion of Lemma~\ref{lem:IminusExclusion} holds for $b=8$ by constructing a specific barrier in the same fashion as in the main proof.

\begin{lemma}\label{lem:IminusExclusionB8}
    For $b=8$ we have $I^-(q)\nsubseteq I^-(p)$. 
\end{lemma}

\begin{proof}
    The barrier we build follows the same pattern. The first piece of the barrier just follows $\partial\Omega^+$ as before. 
    
    For the interpolation piece, we use the same method as in Lemma~\ref{lem:BreachTheGap}: We set the barrier as the graph of $x_{+-}(t)+a(t+0.1)^2$. In order to confirm that this is a barrier for $a$ small enough, we reestablish the earlier discussions for $b=8$. Copying Subsection~\ref{subsec:DefineOmega}, we find that for $b=8$  we have $$\underline\beta(8)\sim 0.101884\mbox{ 
 and }\overline\beta\sim 1.22688.$$ 
    Next we follow Subsection~\ref{subsec:horizontal} and confirm that for $t\in[-0.1,0]$ the estimate $x_{+-}(t)\in\frac 1b(\frac 34,\frac 32)$ still holds. For this we merely need to confirm that for $b=8$ the right hand side of Inequality~\ref{eq:BoxEstimate} satisfies the estimate. Since the function $x_{+-}$ is monotone, we can evaluate that $x_{+-}(-0.1)\equiv 0.0982\equiv 0.786\cdot \frac18$, confirming the lower bound. The upper bound follows from the value of $\underline\beta(8) \equiv 0.101884\equiv 0.815072\frac 18$. 
    With these estimates, we can follow the computations of Subsection~\ref{subsec:interpolation} and confirm that this is indeed a barrier for $a$ small enough by plugging into the right hand side of the inequality~\ref{eq:Criterion3}
$$v^2(8t^2 -(1+t^2)G)+2vxt(-4(b+1)+G)+(2(b+1)^2x^2-(1+x^2)G).$$
If the value is positive for $t\in[-0.1,0]$, then the curve is a non-timelike barrier. 
In Figure~\ref{fig:interpolateB8A} you can see the values for the curve $x_{+-}(t)+a(t+0.1)^2$ for the values $a=1, a=2.66, a=3$ next to each other.  It is clearly visible that for small $a$ (until about 2.66) the function is positive for all $t\in[-0.1,0]$, but for large enough $a$ (after about 2.66) the function becomes negative for some $t$. 
\begin{figure}[hp]
\centering
\includegraphics[width=.3\textwidth]{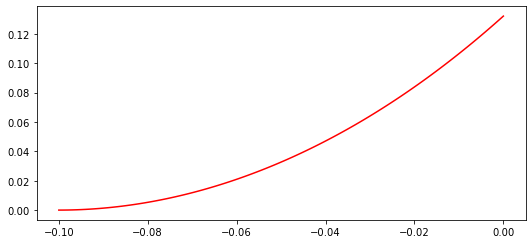}\hfill
\includegraphics[width=.3\textwidth]{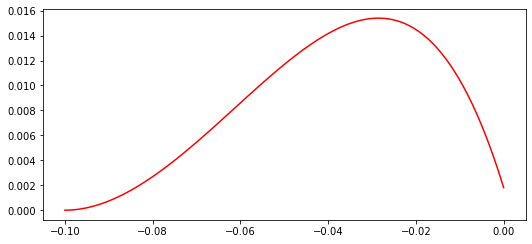}\hfill
\includegraphics[width=.3\textwidth]{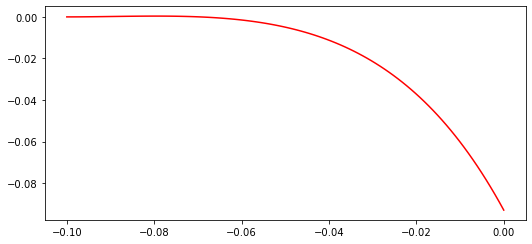}
    \caption{The interpolation for $b=8$ and $a=1, a=2.66, a=3$.}
    \label{fig:interpolateB8A}
\end{figure}

    For the third and last piece we confirm that the graph of $\gamma(t) = \beta \sqrt{1+a x^2}$ from Lemma~\ref{lem:HyperbolicBarrier} is for suitable $a,\beta$ also a barrier $b=8$. To confirm that  $a=6$ and $\beta=0.102$ are suitable, one can plug these values into the right hand side of Inequality~\ref{eq:ConfirmHyperbolicBarrier}. This results in the coefficients $$4.863 t^4 + 0.064 t^2 + 0.002$$ which are all positive, confirming the inequality and thus confirming that the curve with these coefficients constitutes a barrier. Note that $\underline \beta (8)=0.101884$, so $\beta=0.102$ is close to the limit.    
\end{proof}

%-----------------------------------------
%-----------------------------------------
%-----------------------------------------
%-----------------------------------------
%-----------------------------------------
%-----------------------------------------
%-----------------------------------------
%-----------------------------------------
%-----------------------------------------
%-----------------------------------------

\section{Appendix: Gradient flow lines are geodesics}\label{sec:appendix}
The gradient vector field is central to our understanding of this spacetime. In this appendix we will show that the gradient flow lines are geodesics. This knowledge is not necessary for the main construction of this paper, but we imagine that it may be of independent interest.

Note that the gradient vector field is linear $\nabla f(x)=2 diag(a_1,\ldots,a_n)x$, thus its flow is of the form 
$$\varphi_{\nabla f}^t(x)=(e^{2a_1t}x_1,\ldots,e^{2a_nt}x_n).$$ 
A generic flow line is a hyperbola asymptotic to the unstable manifold in negative time and to the stable manifold in positive time\footnote{Note that here the signs are inverted because the names unstable and stable are chosen for the negative gradient flow.} 

\begin{proposition}
    Gradient flow lines of $f$ are geodesics. Furthermore, they are maximizers of $g$-length between level sets of $f$.
\end{proposition}
\begin{proof}
    Let $\gamma:[0,1]\to M$ be a future pointing timelike curve connecting the level sets $\{f=f_0\}$ and $\{f=f_1\}$. Then by the usual Cauchy-Schwartz inequality, we have
    \begin{align*}
        l^g(\gamma)&=\int_0^1 \sqrt{-g(\dot\gamma,\dot\gamma) }dt\\
        &=\int_0^1 \sqrt{-\langle \nabla f,\nabla f\rangle \langle \dot\gamma,\dot\gamma\rangle + \zeta \langle \nabla f,\dot\gamma\rangle^2 }\\
        &\leq \int_0^1 \sqrt{-\langle\nabla f,\dot\gamma\rangle^2 + \zeta \langle \nabla f,\dot\gamma\rangle^2 }\\
        &=\sqrt{\zeta-1} \int_0^1 \langle \dot\gamma,\nabla f\rangle dt\\
        &= \sqrt{\zeta-1} (f_1-f_0),
    \end{align*}
    with equality iff $\dot\gamma$ is almost everywhere parallel to $\nabla f$. This means that gradient flow lines are $g$-length maximizers between two level sets of $f$ and thus in particular timelike geodesics.
\end{proof}

Note that also on the infinitesimal level $g(X_f,X_f)=g(\nabla f,\nabla f)/\|\nabla f\|^4 \equiv 1-\zeta$. Thus, the above proposition implies that
\begin{corollary}
The vector field 
\begin{equation}
X_f=\frac{\nabla f}{\|\nabla f\|^2_h}
\end{equation}
is a $g$ geodesic vector field on $\mathbb{R}^4\setminus\{0\}$ with constant speed $\sqrt{\zeta-1}$.
\end{corollary}

\end{document}